\documentclass[11pt,
final,
english
]{article}

\usepackage[utf8]{inputenc}    
\usepackage[T1]{fontenc}       
\usepackage{longtable}         
\usepackage{calc}
\usepackage{exscale}           
\usepackage{array}             
\usepackage{subcaption} 
\usepackage{wrapfig} 
\usepackage{gitinfo2} 
\usepackage{xspace}            
\usepackage{stackengine} 	
\setstackEOL{\\}
\usepackage{faktor}			   
\usepackage{ifdraft}           
\usepackage{xparse} 
\usepackage{graphicx}   
\usepackage{appendix}
\usepackage{geometry}
\usepackage[multiuser]{fixme}  
\usepackage[expansion=false    
]{microtype}        
\usepackage{fancyhdr}          
\usepackage{xkeyval}           
\usepackage{enumitem}
\usepackage{tikz}              
\usepackage[english,strings]{babel}     
\usepackage[backend = biber,
style = alphabetic,
maxcitenames = 99,
maxbibnames = 99,
mincitenames = 5,
doi=false,
isbn=false,
url=false]{biblatex}
\usepackage[
final,                    
plainpages=false,              
pdfpagelabels=true,            
pdfencoding=auto,              
unicode=true,                  
hypertexnames=true,            
naturalnames=true,              
colorlinks,
linkcolor={blue!70!black},
citecolor={blue!50!black},
urlcolor={blue!80!black}
]{hyperref}                    
\usepackage{nchairx} 
\usepackage{csquotes}

\usetikzlibrary[patterns,patterns.meta,cd]

\usetikzlibrary{cd,arrows,calc,through,backgrounds,matrix,
	decorations.pathmorphing,positioning,babel,arrows.meta,
	decorations.markings}

\graphicspath{{../tikz/pic/}{../tikz/plot/}}

\ifdraft{\synctex=1}{}

\FXRegisterAuthor{md}{anmd}{\color{green!50!black}marvin}
\FXRegisterAuthor{ce}{ance}{chiara}
\FXRegisterAuthor{js}{anjs}{jonas}

\newcommand{\AuthorOne}{\textbf{Marvin Dippell}}
\newcommand{\AuthorTwo}{\textbf{Chiara Esposito}}
\newcommand{\AuthorThree}{\textbf{Jonas Schnitzer}}

\newcommand{\AuthorAddressOne}{
        \begin{minipage}{8cm}
            \centering\small
            Dipartimento di Matematica\\
            Università degli Studi di Salerno\\
            via Giovanni Paolo II, 123\\
            84084 Fisciano (SA)\\
            Italy
        \end{minipage}
        \\
        }
\newcommand{\AuthorAddressTwo}{\AuthorAddressOne}
\newcommand{\AuthorAddressThree}{
        \begin{minipage}{8cm}
            \centering\small
            Dipartimento di Matematica "Felice Casorati"\\
            Università degli Studi di Pavia\\
            Via Ferrata 5\\
            27100 Pavia\\
            Italy
        \end{minipage}
        \\
}

\newcommand{\AuthorEmailOne}{\texttt{mdippell@unisa.it}}
\newcommand{\AuthorEmailTwo}{\texttt{chesposito@unisa.it}}
\newcommand{\AuthorEmailThree}{\texttt{jonaschristoph.schnitzer@unipv.it}}

\author{\AuthorOne\thanks{\AuthorEmailOne}, \,
        \AuthorTwo\thanks{\AuthorEmailTwo},\\[0.2cm]
        \AuthorAddressTwo
        \\[-0.1cm]
        \AuthorThree\thanks{\AuthorEmailThree},\\[0.2cm]
        \AuthorAddressThree
}

\newcommand{\Naturals}{\mathbb{N}}
\newcommand{\Integers}{\mathbb{Z}}
\newcommand{\Reals}{\mathbb{R}}
\newcommand{\ComplexNum}{\mathbb{C}}

\newcommand{\VectFields}{\mathfrak{X}}
\newcommand{\euler}{\mathcal{E}}
\newcommand{\Jmap}{\mathcal{J}}
\newcommand{\poly}{\mathrm{poly}}
\newcommand{\vanishing}{\mathcal{J}}
\newcommand{\normalizer}{\mathcal{N}}
\newcommand{\Ann}{\mathrm{Ann}}

\newcommand{\Cohom}{\operator{H}}
\newcommand{\HC}{\operator{HC}}
\newcommand{\HCdiff}{\HC_{\script{diff}}}
\newcommand{\hkr}{\operator{hkr}}
\newcommand{\hkrInv}{\hkr^{-1}_{\scriptscriptstyle\nabla}}
\newcommand{\hkrHomo}{H_{\scriptscriptstyle\nabla}}

\newcommand{\Def}{\operatorname{Def}}
\newcommand{\homogen}{{\scriptscriptstyle{\mathrm{hom}}}}

\newcommand{\Total}{{\scriptscriptstyle{\mathrm{T}}}}
\newcommand{\Normal}{{\scriptscriptstyle{\mathrm{N}}}}
\newcommand{\proj}{{\scriptscriptstyle{\mathrm{proj}}}}

\newcommand{\dbltilde}[1]{\tilde{\raisebox{0pt}[0.85\height]{$\tilde{#1}$}}}%
\DeclareNameAlias{default}{family-given}
\title{Classification and Reduction of Homogeneous Star Products}

\date{}

\begin{document}

\selectlanguage{english}

\maketitle	

\begin{abstract}
We present a classification of homogeneous star products
on duals of Lie algebroids in terms of the second Lie algebroid cohomology.
Moreover, we extend this classification to projectable star products, i.e. to quantizations compatible with (coisotropic) reduction.
This implies that quantization commutes with reduction in the considered setting.
\end{abstract}

\tableofcontents

\section{Introduction}
\label{sec:Introduction}

Deformation quantization as introduced by Bayen et al. \cite{bayen.et.al:1977a}
aims to formalize the quantization of classical mechanical systems by studying formal deformations of the algebra of smooth 
functions $\Cinfty(M)$ on a given Poisson manifold $(M,\pi)$,
so called star products.
Here $\Cinfty(M)$ is interpreted as the observable algebra on the
classical phase space $(M,\pi)$,
while the algebra $\Cinfty(M)\formal{\hbar}$ equipped with a star product 
$\star$ models the algebra of quantum observables. 
The existence and classification of star products has been studied in many 
different context, see e.g.
\cite{fedosov:1994a,gutt:1983a,dewilde.lecomte:1983a},
and has famously been solved in full generality by Kontsevich
\cite{kontsevich:2003a}.

Mechanical systems are often endowed with first class constraints,
which are modelled by coisotropic submanifolds $C \subseteq (M,\pi)$.
This data allows to reduce the Poisson manifold $(M,\pi)$
to a Poisson manifold $(M_\red,\pi_\red)$ of smaller dimension,
the reduced phase space.

Quantizing a mechanical system with first class constraints poses now the 
following questions:
\begin{cptitem}
	\item Is there a \emph{projectable} star product quantizing $(M,\pi)$,
	i.e. a star product $\star$ which induces a quantization $\star_\red$
	of the reduced phase space $(M_\red,\pi_\red)$?
	\item How can projectable star products be classified?
\end{cptitem}
The question if quantization commutes with reduction in this sense, originally formulated for geometric quantization, has been 
studied in different contexts, e.g. \cite{meinrenken:1996a,guillemin.sternberg:1982a}.

One of the few examples in which it is known how to decide if a star product 
is equivalent to a projectable one is the symplectic case, treated in
\cite{bordemann:2004a:pre,reichert:2017a,kowalzig.neumaier.pflaum:2005a}.
In the situation of general Poisson manifolds a famous counterexample
by Willwacher \cite{willwacher:2007a} shows
that one cannot expect that every Poisson structure admits a
reducible quantization.
The aim of this paper is to investigate the special case of linear
Poisson structures on vector bundles, or equivalently
on duals of Lie algebroids.
A natural choice for quantizations of linear Poisson structures
is given by \emph{homogeneous star products}, which can be seen as star
products compatible with the linear structure.
These star products have been extensively studied in the literature,
see e.g. \cite{neumaier.waldmann:2003a,bordemann.neumaier.waldmann:1998a,bordemann.neumaier.waldmann:1999a}.
It is important to remark that two well-known constructions of star product
on cotangent bundles (see \cite{dewilde.lecomte:1983a,gutt:1983a})  are in fact homogeneous. 
With the help of a recent improvement of the classical 
Hochschild-Kostant-Rosenberg Theorem
\cite{dippell.esposito.schnitzer.waldmann:2024a:pre}
we construct and classify homogeneous 
star products in terms of the underlying second Lie algebroid cohomology,
by assigning to every homogeneous star product $\star$ its \emph{characteristic class $\Phi(\star)$}.
Moreover, we compute the characteristic classes of known constructions.

In view of the classification of homogeneous star products it
seems natural to expect that the projectability of a star product should be 
reflected on its characteristic class.
In fact we find a condition for a star product to be equivalent to a 
projectable one, which turns out to be only a property of its characteristic
class.
To be more precise, the characteristic class of a projectable star product
has to be projectable itself.

However, we observe by an explicit geometric example that two projectable star products, which are
equivalent, can induce non-equivalent star products on the reduced space.
To remedy this, we need to restrict ourselves to
projectable equivalences.
We then provide a classification of projectable star products up to 
projectable equivalences, obtained by the cohomology of a subcomplex of the 
Chevalley-Eilenberg complex of the Lie algebroid.
It turns out that assigning a projectable star product to a given class can be 
understood as quantization.
This raises the question if quantization commutes with reduction in this 
setting.
Indeed, we can answer this question positively, showing that
the reduction of the assigned star product is equivalent to
the quantization of the reduced class.

One remarkable feature of homogeneous star products is that the
(fibrewise- and $\hbar$-)polynomial functions form a subalgebra.
These polynomial function can be evaluated at every complex value of $\hbar$.
Star products possessing this feature are the essential ingredient in the 
study of convergent deformation quantization as treated in
\cite{heins.roth.waldmann:2023a,waldmann:2019a,esposito.stapor.waldmann:2017a,esposito.schmitt.waldmann:2019a}.
Thus our results provide a promising starting point to study reduction in the 
convergent setting.

\paragraph{Organization of the paper}
After collecting some preliminary definitions and results on star products and 
homogeneity on vector bundles in \autoref{sec:Preliminaries} we prove the 
existence of homogeneous star products in \autoref{sec:HomoStarProdExistence}.
Afterwards, we provide a classification in
\autoref{sec:ClassificationHomStarProd} and compare our results with various 
known constructions in \autoref{sec:Comparison}.

In \autoref{sec:ReductionHomStarProd} we turn to projectable star products.
As an intermediate step towards their classification we need to study 
representable star products in \autoref{sec:PresentableStarProd}.
Finally, we provide a full classification of projectable star products in 
\autoref{sec:ClassificationProjStarProd} and
show that reduction commutes with quantization in this setting.

\section*{Acknowledgments}

C.E. was supported by the
National Group for Algebraic and Geometric Structures, and
their Applications (GNSAGA -- INdAM). 
This paper has been supported from Italian Ministerial grant PRIN 2022 "Cluster algebras and Poisson Lie groups", n. 20223FEA2E - CUP D53C24003330006.

\section{Preliminaries}
\label{sec:Preliminaries}

\subsection{Star Products and Hochschild Complex}
\label{sec:StarProducts}

We collect some well-known definitions and results concerning star products and the Hochschild complex for the algebra $\Cinfty(M)$ of smooth complex-valued functions on a manifold $M$.
All of this can be found in the original papers
\cite{gerstenhaber:1963a,gerstenhaber:1964a} or in the text book \cite{waldmann:2007a}.

\begin{definition}[Star Product]
	\label{def:StarProduct}
Let $M$ be a manifold. 
\begin{definitionlist}
	\item A \emph{star product} $\star$ on $M$ is an associative $\ComplexNum\formal{\hbar}$-bilinear product on $\Cinfty(M)\formal{\hbar}$, such that 
	for all $f,g\in \Cinfty(M)$
	\begin{equation}
		f \star g
		= fg + \sum_{r=1}^\infty \hbar^r C_{r}(f,g), 
		\quad\text{and}\quad
		f \star 1 = f = 1 \star f,
	\end{equation}
	where $C_r$ are bidifferential operators for all $r \geq 1$.
	If the associativity of $\star$ only holds up to order $k$ in $\hbar$ we call it a 
	\emph{star product up to order $k$}.
	\item Two star products $\star$ and $\tilde{\star}$ are called \emph{equivalent}, if there exists a series
	$S = \id + \sum_{r=1}^\infty \hbar^r S_r \in \Diffop(M)\formal{\hbar}$,
	such that for all $f,g \in \Cinfty(M)\formal{\hbar}$
	it holds
	\begin{equation}
		\label{eq:EquivalenceStarProd}
		S(f\star g) = S(f) \mathbin{\tilde{\star}} S(g). 
	\end{equation}
	If \eqref{eq:EquivalenceStarProd} holds only up to order $k$ in $\hbar$
	we call $S$ an \emph{equivalence up to order $k$}.
\end{definitionlist}
\end{definition}
In the following we will often denote the pointwise product on $\Cinfty(M)$ by $\mu_0$, thus write
$\star = \mu_0 + \sum_{r=1}^{\infty} \hbar^r C_r$.
For a bidifferential operator $D$ we denote by 
\begin{equation}
	\label{eq:SymAntiSymParts}
	D^+(f,g)
	\coloneqq D(f,g) + D(g,f)
	\quad\text{and}\quad
	D^-(f,g)
	\coloneqq D(f,g) - D(g,f)
\end{equation}
its symmetrization and anti-symmetrization, respectively.
Then an easy check shows that every star product $\star$ induces a Poisson structure on $M$ with Poisson bracket given by
\begin{equation}
	\I \{f,g\} \coloneqq C_1^-(f,g),
\end{equation}
and equivalent star products induce the same Poisson bracket.

Star products can be seen as Maurer-Cartan elements of a differential graded Lie algebra, which we will introduce 
in the following. 
For $n\geq 0$, we  consider the subspace 
\begin{equation}
  \HCdiff^n(M)
	\subseteq 
  \Hom\bigl(\Cinfty(M)^{\tensor n+1},\Cinfty(M)\bigr)
  \eqqcolon \HC^n\bigl(\Cinfty(M)\bigr)
\end{equation}
of polydifferential operators which vanish on constants and for $n=-1$ we set $\HCdiff^{-1}(M)=\Cinfty(M)$. 
The concatenation of two elements $D,E\in \HC^\bullet(M)$
is defined by 
\begin{equation}
  D \circ E (f_0,\dots,f_{k+\ell}) 
  = 
	\sum_{i=0}^{\abs{ D}} (-1)^{i \abs{ E}} 
	D(f_0,\dots, f_{i-1}, E(f_i,\dots,f_{i+\ell}),f_{i+\ell+1},\dots,f_{k+\ell}),
\end{equation}
and the \emph{Gerstenhaber bracket} is the graded commutator with respect to this product 
(up to a sign)
\begin{equation}
  \label{eq:GerstenhaberBracketClassical}
  [D,E]
	=
	(-1)^{\abs{D}\abs{E}} \left(D \circ E - (-1)^{\abs{D} \abs{E}} E \circ D\right).
\end{equation}
This turns $\HC(M)$ into a graded Lie algebra.
and $\HCdiff(M)$ into a graded Lie subalgebra.
Note that we use the sign convention from \cite{bursztyn.dolgushev.waldmann:2012a}, 
not the original one from \cite{gerstenhaber:1963a}.

The pointwise product of functions
$\mu_0 \in \HC^1(\Cinfty(M))$ is not an element of $\HCdiff^1(M)$, 
since it does not vanish on constants, but one can check that 
$\del \colon \HC^\bullet(M) \to \HC^{\bullet +1}(M)$ defined via 
$\del = [\mu_0, \argument]$
is a differential, the so-called \emph{Hochschild differential},
which restricts to $\HCdiff^\bullet(M)$.
Explicitly $\del$ is given by
\begin{equation}
			\label{eq:HochschildDiff}
	\begin{split}
	(\del D)(f_0, \dotsc, f_{n+1})
	&= f_0 \cdot D(f_1, \dotsc, f_{n+1})
	+ (-1)^{n} D(f_0, \dotsc, f_n) \cdot f_{n+1}
	\\
	&\phantom{=}+ \sum_{i=0}^n (-1)^{i+1} D(f_0,\dotsc, f_i \cdot f_{i+1},\dotsc,f_{n+1}),
	\end{split}
\end{equation}
for $D \in \HC^n(M)$ and $f_0, \dotsc, f_{n+1} \in \Cinfty(M)$.
Moreover $\del$ is compatible with the Gerstenhaber bracket and turns 
$(\HCdiff^\bullet(M),\del,[\argument,\argument])$
into a differential graded Lie algebra. 

One can show, see e.g. \cite{waldmann:2007a}, that the associativity of a star product 
$\star = \mu_0 + \sum_{r = 1}^\infty \hbar C_r$
is equivalent to the statement that
\begin{equation}
	\label{eq:MCOrd}
	\del C_{r} +\frac{1}{2}\sum_{\ell=1}^{r-1}[C_\ell,C_{r-\ell}]
	= 0
\end{equation}
holds for all $r \geq 1 $.
Moreover, $\star$ is a star product up to order $k$ if and only if \eqref{eq:MCOrd}
is satisfied for all $r \leq k$.
To construct a star-product one can try to construct it 
order-by-order.  
The following well-known result (see \cite{gerstenhaber:1964a,waldmann:2007a}) 
clarifies, when one can extend a star product up 
to order $k$ to a star product up to order $k+1$.

\begin{lemma}
\label{lem:DefTheorystuff}
	Let
	$\star = \mu_0 +\sum_{r = 1}^k\hbar^r C_r$
	be a star product up to order $k$, then 
	\begin{equation}
		R_{k+1} \coloneqq - \frac 12\sum_{\ell=1}^{k}[C_\ell,C_{k+1-\ell}]
	\end{equation}
	is closed, i.e. $\del R_{k+1} = 0$.
	Moreover, if there exist a $C_{k+1}\in \HCdiff^1(M)$, such that 
	$\del C_{k+1}
	= R_{k+1}$,
	then 
	$\mu_0 + \sum_{r= 1}^{k+1}\hbar^r C_r$
	is a star product up to order $k+1$. 
\end{lemma}
This lemma shows that the cohomology with respect to $\del$ is of clear 
interest for the order-by-order construction of a star product. 
Fortunately, the well-known Hochschild-Kostant-Rosenberg Theorem \cite{hochschild.kostant.rosenberg:1962a} computes the Hochschild cohomology.
We will need a slightly stronger result, giving a deformation retract for the Hochschild complex, see \cite{dippell.esposito.schnitzer.waldmann:2024a:pre}.
To formulate this we introduce a map $\hkr$ connecting multivector fields 
$\VectFields^\bullet(M) \coloneqq \Secinfty(\Anti^{\bullet}TM)$
to polydifferential operators.
This map 
$\hkr \colon \VectFields^{\bullet+1}(M) \to \HCdiff^\bullet(M)$
is called the \emph{Hochschild-Kostant-Rosenberg map} and is defined by 
\begin{equation}
	\hkr(X_0 \wedge \dots \wedge X_k)(f_0, \dots, f_k)
	\coloneqq
	\frac{1}{(k+1)!} \sum_{\sigma\in S_{k+1}}
	\sign(\sigma) \cdot 
	\Lie_{X_{\sigma(0)}}(f_0) \cdots \Lie_{X_{\sigma(k)}}(f_k).
\end{equation}	  
In fact this map is actually a chain map (even a quasi-isomorphism), if we endow $\VectFields^\bullet(M)$ with the trivial differential.
Note that the map $\hkr$ is by no means a morphism of differential graded Lie algebras.

\begin{theorem}[Hochschild-Kostant-Rosenberg Theorem \cite{dippell.esposito.schnitzer.waldmann:2024a:pre}]
    \label{thm:classicalHKR}%
    Let $M$ be a manifold and let $\nabla$ be a torsion-free covariant
    derivative on $M$.  Then there exist
	$\hkrHomo \colon \HCdiff^\bullet(M) \to \HCdiff^{\bullet-1}(M)$ and 
	$\hkrInv \colon \HCdiff^\bullet(M) \to \VectFields^{\bullet+1}(M)$
	such that
    \begin{equation}
        \label{eq:HKRDeformationRetract}
        \begin{tikzcd}[column sep = large]
            \VectFields^{\bullet+1}(M)
            \arrow[r,"\hkr", shift left = 3pt]
            &\bigl( \HCdiff^{\bullet}(M),\del \bigr)
            \arrow[l,"\hkrInv", shift left = 3pt]
            \arrow[loop,
            out = -30,
            in = 30,
            distance = 30pt,
            start anchor = {[yshift = -7pt]east},
            end anchor = {[yshift = 7pt]east},
            "\hkrHomo"{swap}
            ]
        \end{tikzcd}
    \end{equation}
    is a deformation retract, i.e. the following holds:
    \begin{theoremlist}
		\item \label{thm:classicalHKR_1}
			We have $\hkrInv \circ \hkr = \id$.
		\item \label{thm:classicalHKR_2}
			We have
			$\del \hkrHomo + \hkrHomo \del
			= \id - \hkr \circ \hkrInv$.			
    \end{theoremlist}
	Moreover, if the connection is invariant under a Lie algebra action,
	the maps $\hkrInv$ and $\hkrHomo$ 
	are equivariant with respect to this action. 
\end{theorem}
Using the Hochschild-Kostant-Rosenberg Theorem, we can even refine a bit the 
methods to check, if a closed element is exact. 

\begin{lemma}
\label{lem:Skewstuff}
	Let $D\in \HCdiff^k(M)$ be closed, i.e. $\del D = 0$. 
	\begin{lemmalist}
		\item The total anti-symmetrization $\AntiSymmetrizer(D)$ of $D$, defined by
		\begin{equation}
			\AntiSymmetrizer(D)(f_1, \dotsc, f_{k+1})
			\coloneqq \frac{1}{(k+1)!}\sum_{\sigma \in S_{k+1}} \sign(\sigma) \cdot 
			D(f_{\sigma(1)},\dotsc, f_{\sigma(k+1)}),
		\end{equation}
		is given by
		$\AntiSymmetrizer(D) = \hkr \hkrInv D$. 
		\item If $\AntiSymmetrizer(D) = 0$, then $D$ is exact with
		$D = \del \hkrHomo D$.  
	\end{lemmalist}
\end{lemma}

\begin{proof}
Since $D$ is closed, we have by \autoref{thm:classicalHKR} that 
$D = \hkr \hkrInv D+ \del \hkrHomo D$.
It is well-known that $\mathrm{Alt} \circ \del = 0$, see e.g. \cite{waldmann:2007a}, and therefore we get 
$\mathrm{Alt}(D)=\mathrm{Alt}(\hkr \hkrInv D)=\hkr \hkrInv D$,
where we used in the last step that the image of the Hochschild-Kostant-Rosenberg map is already 
totally anti-symmetric.
The second statement then follows by using again the homotopy formula from \autoref{thm:classicalHKR}. 
\end{proof}

\subsection{Homogeneity on Vector Bundles}
\label{sec:HomogeneityOnVB}

Let $\pr \colon E \to M$ be a vector bundle.
We define the so-called
\emph{Euler vector field} field
$\euler \in \VectFields(E)$
by its flow as
\begin{equation}
	\euler\at{e_p}=\tfrac{\D}{\D t}\at{t=0} \E^t e_p
\end{equation}
for $p \in M$ and $e_p \in E_p$.
A function $F \in \Cinfty(E)$ is now called \emph{homogeneous of degree $k$},
if it is an eigenvector of $\Lie_\euler$ to the eigenvalue $k$, i.e.
\begin{equation}
	\Lie_\euler F = k F.
\end{equation}
We denote by $\Pol^k(E) \subseteq \Cinfty(E)$ the subspace of all homogeneous functions of degree
$k$.
Note that $F$ is homogeneous of degree $0$ if and only if it is given by
$F = \pr^*f$ for some $f \in \Cinfty(M)$.

Alternatively, we can characterize homogeneous functions using the map
$\Jmap \colon \Secinfty(E^*) \to \Cinfty(E)$,
which is defined pointwise by  
$\Jmap(\alpha)(e_p) = \alpha_p(e_p)$, 
for all $p \in M$ and $e_p \in E_p$.
After extending $\Jmap$ to an algebra morphism
$\Jmap \colon \Secinfty(\Sym E^*) \to \Cinfty(E)$
one can show that the image of $\Secinfty(\Sym^k E^*)$ under $\Jmap$ is exactly given by
$\Pol^k(E)$.

Generalizing this idea we say that a multivector field $X \in \VectFields^n(E)$ is
\emph{homogeneous of degree $k$}
if it satisfies
\begin{equation}
	\Lie_\euler X = k X.
\end{equation}
An important observation is that the possible degrees of homogeneity depend on the degree of the multivector field.

\begin{proposition}
Let $E \to M$ be a vector bundle.
\begin{propositionlist}
	\item Let $F \in \Pol^k(E)$ be non-zero and homogeneous of degree $k \in \Reals$.
	Then $k \in \Naturals_0$.
	\item If $X \in \VectFields^n(E)$ be non-zero and homogeneous of degree $k \in \Reals$.
	Then $k \in \Integers$ and $k \geq -n$.
\end{propositionlist}
\end{proposition}
A differential operator $D \in \HCdiff^\bullet(E)$
is called \emph{homogeneous of degree $k$}
if it satisfies
\begin{equation}
	[\Lie_\euler,D] = k D.
\end{equation}
Here we use the Gerstenhaber bracket introduced above.
Note that the interpretation of a multivector field
$X \in \VectFields^n(E)$
as a multi-differential operator in
$\HCdiff^{n-1}(E)$
is compatible with the homogeneous degrees.

To a section $s \in \Secinfty(E)$ we can associate a vector field
$s^\ver \in \Secinfty(TE)$, the \emph{vertical lift}, which is given by its flow 
\begin{align}
	s^\ver(e_p)= \frac{\D}{\D t}\at{t=0} (e_p+ts(p)).
\end{align}
Note that this vector field has the properties
\begin{equation}
	\label{eq:VerticalLiftProps}
\begin{aligned}
	(fs)^\ver
	&= \pr^*f\cdot s^\ver,
	&
	\Lie_{s^\ver}(p^*f)
	&= 0,
	\\
	[s^\ver,t^\ver] 
	&= 0,
	&
	\Lie_{s^\ver}(\mathcal{J}(\alpha))
	&= \pr^*\alpha(s),
\end{aligned}
\end{equation}
for all $f\in \Cinfty(M)$,
$s,t \in \Secinfty(E)$
and $\alpha\in \Secinfty(E^*)$.
Therefore, we have that
$[\euler,s^\ver]=-s^\ver$.
Moreover, we can extend the vertical lift 
to multivector fields
$\cdot^\ver \colon \VectFields^\bullet(E)
\to \VectFields^\bullet (TE)$.

A covariant derivative
$\nabla \colon \Secinfty(TE) \tensor \Secinfty(TE) \to \Secinfty(TE)$
is called \emph{homogeneous} if it satisfies
\begin{equation}
	\label{eq:HomoCovariantDerivative}
	\Lie_\euler (\nabla_X Y) 
	= \nabla_{\Lie_\euler X} Y 
		+ \nabla_X(\Lie_\euler Y),
\end{equation}
for all $X,Y \in \Secinfty(TE)$.

\begin{proposition}[Homogeneous HKR Theorem]
	\label{prop:HomoHKRTheorem}
	Let $E \to M$ be a vector bundle.
	\begin{propositionlist}
		\item There exists a torsion-free homogeneous covariant derivative on $E$.
		\item Given a torsion-free homogeneous covariant derivative on $E$
		the maps 
		$\hkrHomo \colon \HCdiff^\bullet(E) \to \HCdiff^{\bullet-1}(E)$
		and 
		$\hkrInv \colon \HCdiff^\bullet(E) \to \VectFields^{\bullet+1}(E)$
		from \autoref{thm:classicalHKR}
		preserve the degrees of homogeneity.
	\end{propositionlist}
\end{proposition}

\begin{proof}
	The defining Equation \eqref{eq:HomoCovariantDerivative}
	of a homogeneous covariant derivative is equivalent to saying that
	$\nabla$ is invariant under the action of the Lie subalgebra of $\VectFields(E)$
	generated by the Euler vector field $\euler$. Its flow acts properly if we restrict it to 
	$E\setminus\iota(M)$
	and thus we can find an invariant covariant derivative with respect to this action. 
	In a local trivialization, one can see that the Christoffel symbols are only 
	constant or linear along the fibres and thus can be extended to the zero section in 
	a unique way. 
	Then the second part follows from \autoref{thm:classicalHKR}.
\end{proof}

\subsection{Lie Algebroids and their Connections}
\label{sec:LieAlgebroids}

We will need some basic properties of Lie algebroids.
All of the following can be found e.g. in \cite{crainic.fernandes:2011a}

\begin{definition}[Lie Algebroid]
	\label{def:LieAlgebroid}
	A \emph{Lie algebroid} consists of a vector bundle
	$A \to M$ together with
	a vector bundle morphism
	$\rho_A \colon A \to TM$ covering the identity, called \emph{anchor},
	and a Lie algebra structure $[\argument, \argument]_A$
	on $\Secinfty(A)$, such that
	\begin{equation}
		\label{eq:LieAlgebroid}
		[s,ft]_A = \rho_A(s)(f)t + f[s,t]_A
	\end{equation}
	holds for all $f \in \Cinfty(M)$ and $s,t \in \Secinfty(A)$.
\end{definition}
To every Lie algebroid there is an associated cochain complex, its so-called \emph{Lie algebroid complex},
given by $\Secinfty(\Anti^\bullet A^*)$ with differential defined by 
\begin{align}
	\label{eq:deRhamDifferential}
	(\D_A \alpha)(s_1,\dotsc, s_{k+1})
	= &\sum_{i=1}^{k+1} (-1)^{i+1}
		\rho_A(s_i)\bigl(\alpha(s_1, \dotsc, \overset{i}{\wedge}, \dotsc, s_{k+1}) \bigr)
	\\
		&+ \sum_{1\leq i<j \leq {k+1}}(-1)^{i+j+1}
		\alpha\bigl([s_i, s_j]_A, s_1,
		\dotsc,\overset{i}{\wedge},\dotsc,\overset{j}{\wedge},\dotsc,
		s_{k+1}\bigr),
\end{align}
for all $\alpha \in \Secinfty(\Anti^k A^*)$ and $s_1, \dots, s_{k+1} \in \Secinfty(A)$,
where $\overset{i}{\wedge}$ denotes the omission of the $i$-th entry.
The corresponding cohomology will be denoted by $\Cohom^\bullet(A)$.

Any Lie algebroid
$A \to M$
carries an associated Poisson structure $\pi_\KKS$
on its dual $A^*$.
This Poisson structure is uniquely determined by 
\begin{equation}
	\label{eq:KKSPoisson}
	\{\pr^*f, \pr^*g\}_\KKS = 0,
		\quad
	\{\Jmap(s),\Jmap(t)\}_\KKS
	= - \Jmap([s,t]_A)
		\quad \text{ and } \quad
	\{\pr^*f, \Jmap(s)\}_\KKS
	= \pr^*\big( \rho_A(s)(f) \big),
\end{equation}
for all
$s,t \in \Secinfty(A)$ and
$f \in \Cinfty(M)$.
Here $\pr \colon A^* \to M$ denotes the bundle projection
and $\Jmap \colon \Secinfty(A) \to \Cinfty(A^*)$
was introduced in \autoref{sec:HomogeneityOnVB}.
Note that this Poisson structure fulfils
$\Lie_\euler \pi_\KKS = -\pi_\KKS$, 
i.e. it is linear.
On the other hand, any linear Poisson structure on $A^*$ 
induces a Lie algebroid on $A$ by using \eqref{eq:KKSPoisson}
to define the Lie bracket and the anchor.

\begin{definition}
Let $(A \to M,[\argument,\argument]_A,\rho_A)$ be a Lie algebroid.
An \emph{$A$-connection} is a map 
$\nabla \colon \Secinfty(A) \tensor \Secinfty(A) \to \Secinfty(A)$, such that
\begin{equation}
	\nabla_{fs}t = f \nabla_s t 
	\qquad\text{and}\qquad
	\nabla_s (ft) = \rho_A(s)(f)t + f\nabla_s t
\end{equation}	 
hold for all $s,t \in \Secinfty(A)$ and $f \in \Cinfty(M)$ 
We say that $\nabla$ is \emph{torsion-free}, if
$\nabla_s t - \nabla_s t = [s,t]_A$ for all $s,t \in \Secinfty(A)$. 
\end{definition}
To each connection on a Lie algebroid $(A \to M,[\argument,\argument]_A,\rho_A)$, we can associate its curvature 
$R^\nabla \in \Secinfty(\Anti^2 A^* \tensor \End(A))$ by
\begin{align}
	R^\nabla(s,t)u
	\coloneqq \bigl([\nabla_s,\nabla_t]-\nabla_{[s,t]}\bigr)u
\end{align}
for $s,t,u\in \Secinfty(E)$.
By choosing a fibre metric we can always find a torsion-free $A$-connection
with
\begin{align}
	\tr(R^\nabla) = 0
	\in \Secinfty(\Anti^2 A^*),
\end{align}
where $\tr$ denotes the fibrewise trace on the $\End(E)$-part.
With an $A$-connection we can obtain a horizontal lift of a section
$s \in \Secinfty(A)$ to a vector field $s^\hor \in \Secinfty(TA^*)$ by 
\begin{align}
	s^\hor(\mathcal{J}(t))
	=\mathcal{J}(\nabla_st)
	\quad \text{ and } \quad 		
	s^\hor(p^*f) 
	= p^*(\rho(s)(f))
\end{align}
for all $t \in \Secinfty(A)$ and
$f \in \Cinfty(M)$.
One can check (e.g. in local coordinates) that this actually defines a vector field.
Moreover, one gets that 
$[\euler,s^\hor] = 0$
for all $s \in \Secinfty(A)$.

\section{Homogeneous Star Products}
\label{sec:HomStarProducts}

In this paper we want to study star products on the total space of a given vector bundle
$\pr \colon E \to M$ which are compatible with the linear structure of $E$.

\begin{definition}[Homogeneous Star Product]\
	\label{def:HomoStarProduct}
Let $E \to M$ be a vector bundle.
\begin{definitionlist}
	\item A star product
	$\star = \mu_0 + \sum_{r=1}^{\infty} \hbar^r C_r$
	on $E$ is called \emph{homogeneous}, if 
	for all $r \geq 1$ the differential operator $C_r$ is homogeneous of degree $-r$, i.e. we have
	\begin{equation}
		\label{eq:HomogeneityCr}
		[\Lie_{\euler}, C_r] = -r \cdot C_r.
	\end{equation}
	\item We say that an equivalence
	$S = \id + \sum_{r=1}^{\infty} \hbar^r S_r$ 
	of star products is \emph{homogeneous}, if for all $r \geq 1$ the differential operator $S_r$
	is homogeneous o degree $-r$, i.e. we have
	\begin{equation}
		\label{eq:HomogeneitySr}
		[\Lie_\euler, S_r] = -r \cdot S_r.
	\end{equation}
\end{definitionlist}
\end{definition}

\begin{remark}
	A star product $\star$ is homogeneous if and only if 
	$\hbar \frac{\del}{\del \hbar} + \Lie_\euler$ is a derivation of $\star$.
	Thus our definition agrees with the definition used in 
	\cite{neumaier.waldmann:2009a}.
\end{remark}
Note that for a homogeneous equivalence $S$ of star products $\star$ and $\tilde{\star}$
we immediately get
$S(\pr^*f) = \pr^*f$ for all $f \in \Cinfty(M)$.
Moreover, if $\star$ is homogeneous then $\tilde{\star}$ is homogeneous as well.
From \eqref{eq:HomogeneityCr} it follows immediately that
\begin{equation}
	\label{eq:CkOnPolynomials}
	C_r\bigl(\Pol^s(E), \Pol^t(E)\bigr)
		\subseteq
	\Pol^{s + t - r}(E).
\end{equation}

\begin{remark}
	Note that by \eqref{eq:CkOnPolynomials} for every homogeneous star product we get that
	$\Pol(E)[\hbar] \subseteq \Cinfty(E)\formal{\hbar}$ is a subalgebra 
	where we can actually evaluate the formal parameter $\hbar$ at any complex value.
\end{remark}
We know that every star product on a manifold induces a Poisson structure on it, see \autoref{sec:StarProducts}.
It is well-known that for a homogeneous star product on the total space of a vector bundle we get the following refined result.

\begin{lemma}
	\label{lem:StarProdInducesLinearPoisson}
	Let $A \to M$ be a vector bundle.
	Let $\star$ be a homogeneous star product on $A^*$, then the associated Poisson structure $\pi \in \VectFields^2(A^*)$ is linear, or, equivalently, $A$ is a Lie algebroid.  
\end{lemma}

\begin{proof}
	By definition of homogeneous star product we know that
	$C_1$ is homogeneous of degree $-1$,
	and thus the induced Poisson structure
	$\{F,G\}_\KKS = C_1^-(F,G)$
	is homogeneous of degree $-1$.
	Equivalently, the corresponding bivector field
	$\pi$ is homogeneous of degree $-1$ and hence
	is a linear Poisson structure.
\end{proof}
\begin{example}
	\label{ex:HomgeneousStarProds}
Many explicit constructions of star products actually canonically produce 
homogenous star products.
We refer to \autoref{sec:Comparison} for a more 
exhaustive discussion, but let us discuss already now some special cases here. 
\begin{examplelist}
	\item \label{ex:HomgeneousStarProds_Moyal}
	We consider $T^*\mathbb{R}^d$ with coordinates $(q^i,p_i)$.
	Then the \emph{Weyl-Moyal} product 
	\begin{align}
		f\star g \coloneqq \sum_{k=0}^\infty \frac{(\I\hbar)^k}{2^k k!} 
		\frac{\partial^k f}{\partial q^{i_1}\dots\partial q^{i_k}}
		\frac{\partial^k f}{\partial p_{i_1}\dots\partial p_{i_k}}
	\end{align}
	which is homogeneous.
	\item \label{ex:HomgeneousStarProds_Cotangent}
	More generally, for a manifold $M$, one can construct homogeneous
	star products on $T^*M$ via the method of Fedosov.
	This is carried out in great detail in 
	\cite{bordemann.neumaier.waldmann:1998a,bordemann.neumaier.waldmann:1999a}. 
	\item \label{ex:HomgeneousStarProds_Gutt}
	Let $\liealg{g}$ be a finite dimensional real Lie algebra, then in 
	\cite{gutt:1983a} 
	it is shown, that one can quantize the corresponding Poisson manifold 
	$(\liealg{g}^*,\pi_\KKS)$ via the universal property of the universal 
	enveloping algebra  $\mathcal{U}(\liealg{g})$.
	This, in fact, results in a homogenous star product.  
\end{examplelist}
\end{example}

\begin{convention}
	From now on, when we consider a homogeneous star product $\star$ on the dual of a Lie algebroid $A$ we will always assume that the linear Poisson structure on $A^*$ induced by $\star$ agrees with the linear Poisson structure induced by $A$.
	Or equivalently, that the Lie algebroid structure on $A$ is the one induced by $\star$.
	As for any Lie algebroid we will denote this Poisson structure by $\pi_\KKS$ or $\{\argument,\argument\}_\KKS$.
	For a fixed Lie algebroid $A$ we will denote the set of homogeneous equivalence classes of homogeneous
	star products deforming the Poisson structure $\pi_\KKS$ by 
	$\Def_\homogen(A^*)$.
\end{convention}

From the above result we know that for every homogeneous star product $\star$
the anti-symmetric part of the first order is given by
$C_1^- = \I \{\argument, \argument\}_\KKS$.
The next result shows that we can in fact always achieve that 
the whole first order $C_1$ is given by the linear Poisson bracket

\begin{lemma}
	\label{lem:EquivHomoStarProducts}
	Let $A \to M$ be a Lie algebroid.
	Every homogeneous star product $\star$ on $A^*$
	is homogeneously equivalent to a homogeneous star product
	$\tilde{\star}$ of the form
	\begin{equation}
		F \mathbin{\tilde{\star}} G 
		= F \cdot G + \frac{\I\hbar}{2} \{F,G \}_\KKS + \hbar^2(\dots).
	\end{equation}
\end{lemma}

\begin{proof}
Let $\nabla$ be a homogeneous covariant derivative on $A^*$.
By \autoref{prop:HomoHKRTheorem}
we know that $\hkrHomo(C_1^+)$
is homogeneous of degree $-1$ and thus
$S \coloneqq \id - \frac{\hbar}{2} \hkrHomo(C_1^+)$ is a homogeneous equivalence.
Define
$F \mathbin{\tilde{\star}} G \coloneqq S^{-1}\bigl(S(F) \star S(G)\bigr)$.
Then in first order of $\hbar$ we obtain
\begin{equation*}
	\label{eq:EquivHomoStarProducts}
	\tag{$*$}
	\tilde{C}_1
	= C_1 - \frac{1}{2} \del\hkrHomo(C_1^+)
\end{equation*}
using \eqref{eq:HochschildDiff}.
Since $\star$ is a star product, we know that
$0 = \del C_1
	= \del(\hkr(\pi_\KKS)+C_1^+)
	= \del C_1^+$.
Thus applying the homotopy formula from \autoref{thm:classicalHKR}
and using \autoref{lem:Skewstuff}
we get
$C_1^+
	= \hkr(\hkrInv(C_1^+)) + \del \hkrHomo(C_1^+)
	=\del \hkrHomo(C_1^+)$.
Inserting this into \eqref{eq:EquivHomoStarProducts} we obtain
\begin{equation*}
	\tilde{C}_1
	= C_1 - \frac 12 C_1^+
	= \frac 12 C_1^-
	= \frac{\I}{2} \{\argument,\argument\}_\KKS,
\end{equation*}
and thus $\tilde{\star}$ is a star product of the expected form.
\end{proof}
Let us collect some useful formulas that we will need later on.

\begin{lemma}
	\label{lem:StarExpan}
	Let $A \to M$ be a Lie algebroid and let
	$\star = \mu_0 + \sum_{r=1}^{\infty} \hbar^r C_r$
	be a homogeneous star product on
	$\pr \colon A^* \to M$.
	Then it holds 
	\begin{align}
		\pr^* f \star \pr^*g
		&= \pr^*(fg),
		\\
		[\Jmap(s),\pr^*f]_\star
		&= \I\hbar\{\Jmap(s),\pr^*f\}_\KKS
		= \I \hbar \pr^*\rho_A(s)(f)
		,
		\\
		\shortintertext{and}
		[\Jmap(s),\Jmap(t)]_\star 
		&= \I\hbar \{\Jmap(s),\Jmap(t)\}_\KKS
			+ \hbar^2 C_2^-\big(\Jmap(s),\Jmap(t)\big)
		= - \I\hbar\Jmap([s,t]_A) 
			+ \hbar^2 C_2^-\big(\Jmap(s),\Jmap(t)\big),
	\end{align}
	for all $f,g \in \Cinfty(M)$ and $s,t \in \Secinfty(A)$.
\end{lemma}

\begin{proof}
	The first part follows since we know that $C_r(\pr^*f,\pr^*g) \in \Cinfty(A^*)$ is homogeneous of degree $-r$ and therefore vanishes.
	For the second and third part, note that
	\begin{equation*}
		[F,G]_\star
		= \I \hbar \{F,G\}_\KKS
			+ \sum_{r=2}^{\infty} \hbar^r C_r^-(F,G)
	\end{equation*}
	for all $F,G \in \Cinfty(A^*)$.
	Now choosing $F = \Jmap(s)$ and $G = \pr^*f$ for some
	$s \in \Secinfty(A)$ and $f \in \Cinfty(M)$
	we know that $C_r(\Jmap(s),\pr^*f)$
	is homogeneous of degree $(r-1)$ and hence vanishes for all $r \geq 2$.
	The first part then follows from \eqref{eq:KKSPoisson}.
	Similarly, for the last part we choose $F = \Jmap(s)$
	and $G = \Jmap(t)$ for $s,t \in \Secinfty(A^*)$.
	Then $C_r(\Jmap(s),\Jmap(t))$ is homogenous of degree $r-2$
	and hence vanishes for $r \geq 3$.
	Again using \eqref{eq:KKSPoisson} we obtain the third part.
\end{proof}

\subsection{Existence}
\label{sec:HomoStarProdExistence}

Recall from \autoref{sec:StarProducts}
that the problem of extending star products is controlled by the Hochschild cohomology.
In the case of homogeneous star products we can exploit the homogeneous HKR Theorem, see \autoref{prop:HomoHKRTheorem},
to obtain simplified statements.

\begin{proposition}
	\label{prop:ExtHomStar}
Let $E \to M$ be a vector bundle
and let $\star = \mu_0 + \sum_{r=1}^{\infty}\hbar^r C_r$
be a homogeneous star product up to order $3$ on $E$.
Then $\tilde\star = \mu_0 + \sum_{r=1}^{\infty} \hbar^r \tilde{C}_r$ given by
\begin{align}
	\tilde{C}_r
	\coloneqq
	\begin{cases}
		C_r	&\text{for } r = 1,2,3, \\
		-\frac{1}{2}\hkrHomo \bigl(\sum_{\ell=1}^{r-1}[\tilde{C}_\ell,\tilde{C}_{r-\ell}]\bigr)	&\text{for } r \geq 4,
	\end{cases}
\end{align} 
is a homogeneous star product on $E$.
\end{proposition}
\begin{proof}
By assumption $\mu_0 + \sum_{r=1}^{3}\hbar^r C_r$ is a homogeneous star product up to order $3$.
We now construct $\tilde{\star}$ by induction.
Thus let us assume we have shown that
$\mu_0 + \sum_{r=1}^{k-1} \hbar^r \tilde{C}_r$ is a homogeneous a star product up to order $k-1 \geq 3$.
Then we have by \autoref{lem:DefTheorystuff} that
$R_{k} \coloneqq -\frac 12 \sum_{\ell=1}^{k-1}[\tilde{C}_\ell,\tilde{C}_{k-\ell}]$
is closed.
Thus after choosing a homogeneous connection $\nabla$ on $E$, we can 
use the deformation retract from \autoref{thm:classicalHKR} to obtain 
\begin{equation*}
	R_{k}
	= \hkr \hkrInv(R_{k}) + \del \hkrHomo (R_{k}).
\end{equation*}
Note that $\hkrInv(R_{k})$ is a tri-vector field with homogeneous degree 
$-k < -3$ and therefore has to vanish.
Thus by \autoref{lem:DefTheorystuff} we know that by setting
$\tilde{C}_{k}:= \hkrHomo (R_{k})$,
we get a star product up to order $k$.
Note that $\tilde{C}_{k}$ is homogeneous of degree $-k$ by \autoref{prop:HomoHKRTheorem}.
This way we obtain a homogeneous star product $\tilde{\star}$.
\end{proof}
We will now give a concrete construction of a homogeneous star product
on the dual of an arbitrary Lie algebroid.
Note that by the above result we only need to construct the second and third order of a homogeneous star product in a compatible way.
Our construction depends on a chosen closed Lie algebroid two-form $B$.
The introduction of this two-form might seem unmotivated at the moment, but its significance will become clear in \autoref{sec:ClassificationProjStarProd}.

\begin{proposition}[Existence of homogeneous star products]
	\label{prop:ExistenceHomStarProd}
	Let $A \to M$ be a Lie algebroid, let $\nabla$ be a homogeneous covariant derivative on $A^*$
	and let $B \in \Secinfty(\Anti^2 A^*)$ be closed.
	Then $\star_{\scriptscriptstyle(\nabla,B)} = \mu_0 + \sum_{r=1}^{\infty} \hbar^r C_r$ defined by
	\begin{align}
		C_1 
		&\coloneqq
		\frac{\I}{2} \{\argument,\argument\}_\KKS,
		\\
		C_2 
		&\coloneqq -\frac{1}{4} \left(\hkrHomo[C_1,C_1]\right)^+
		+ \hkr B^\ver,
		\\
		\shortintertext{and}
		C_r
		&\coloneqq
		-\frac{1}{2}\hkrHomo (\sum_{\ell=1}^{r-1}[C_\ell,C_{r-\ell}]),
	\end{align}
	for all $r \geq 3$,
	is a homogeneous star product on $A^*$.
\end{proposition}

\begin{proof}
	By \autoref{prop:ExtHomStar} it is enough to show that
	$\mu_0 + \sum_{r=1}^{3}\hbar^r C_r$
	is a homogeneous star product up to order $3$.
	For this, first note that all $C_r$ are homogeneous bidifferential operators 
	of degree $-r$.
	It remains to show that $\star_{\scriptscriptstyle(\nabla,B)}$
	is associative up to order $3$.
	In other words, that \eqref{eq:MCOrd} is satisfied for $r = 1,2,3$.
	Since $\{\argument,\argument\}_\KKS$ is a biderivation we have
	$\del C_1 = 0$.

	For the second degree we compute
	\begin{equation*}
		\del C_2 
		= - \frac 14 \del \left(\hkrHomo[C_1,C_1]\right)^+
		+ \del \hkr B^\ver
		= - \frac 14 \del \left(\hkrHomo[C_1,C_1]\right)^+.
	\end{equation*}
	Define
	$(\tau D)(F,G) \coloneqq D(G,F)$
	for any $D \in \HCdiff^1(A^*)$,
	then $\del(\tau D)(F,G,H) = - (\del D)(H,G,F)$
	follows from the definition of $\del$, 
	see \eqref{eq:HochschildDiff}. 
	Together with $[C_1, C_1] = \del \hkrHomo[C_1,C_1]$,
	see \autoref{lem:Skewstuff}, we obtain
	\begin{align*}
		\del \bigl(\tau(\hkrHomo[C_1,C_1])\bigr)(F,G,H)
		= - \del \hkrHomo[C_1,C_1](H,G,F)
		= - [C_1,C_1] (H,G,F)
		= [C_1,C_1](F,G,H),
	\end{align*}
	where the last equality holds since $[C_1,C_1]$ is anti-symmetric in the first and last slot.
	Thus we get
	\begin{align*}
		- \frac 14 \del \left(\hkrHomo[C_1,C_1]\right)^+
		= - \frac 14 \del \left(\hkrHomo[C_1,C_1] + \tau (\hkrHomo[C_1,C_1]) \right)
		= - \frac 12 [C_1,C_1],
	\end{align*}
	and therefore $\star_{\scriptscriptstyle(\nabla,B)}$
	is associative up to order $2$.

	Finally, to consider associativity in third order 
	consider $C_3 = - \hkrHomo[C_2,C_1]$.
	Note that we have
	$\AntiSymmetrizer \bigl( [(\hkrHomo[C_1, C_1])^+,C_1 ]\bigr)= 0$
	since $(\hkrHomo[C_1,C_1])^+$ is symmetric and
	$C_1$ is anti-symmetric.
	Moreover, for $s,t,u \in \Secinfty(A)$ we get
	\begin{align*}
		[\hkr B^\ver, C_1]
		= - \frac{\I}{2} \pr^*
			\Bigl(
				B([s,t]_A,u)
				+ B([t,u]_A,s)
				+ \rho_A(u)B(s,t)
				+ \rho_A(s)B(t,u)\Bigr).
	\end{align*}
	Then $\AntiSymmetrizer\bigl([\hkr B^\ver,C_1]\bigr)
	= -\I\pr^* \D_A B = 0$.
	Together this shows that $\AntiSymmetrizer([C_2,C_1]) = 0$ and therefore $\del\hkrHomo[C_2,C_1] = [C_2,C_1]$
	by \autoref{lem:Skewstuff}.
	Finally, this shows
	$\del C_3 = -[C_2,C_1] = - \frac 12 \sum_{\ell=1}^{2}[C_\ell,C_{3-\ell}]$.
	And hence we obtain a full homogeneous star product by
	\autoref{prop:ExtHomStar}.
\end{proof}

\subsection{Classification}
\label{sec:ClassificationHomStarProd}

The goal of this chapter is to provide a classification of homogeneous star products
on the dual of a fixed Lie algebroid.
More specifically, we will construct a bijection
$\Phi \colon \Def_\homogen(A^*) \to \Cohom^2(A)$
between equivalence classes of homogeneous star products on $A^*$
and the second Lie algebroid cohomology of $A$.

\begin{proposition}
	\label{prop:EquivHomStar}
	Let $E \to M$ be a vector bundle.
	Moreover, let $\star$ and $\tilde{\star}$ be homogenous star products on $E$.
	If $\star$ and $\tilde{\star}$ are homogenously equivalent up to order $2$, then
	they are homogenously equivalent.
\end{proposition}
	
\begin{proof}
	Let $\star = \mu_0 + \sum_{r=1}^{\infty} \hbar^r C_r$
	and $\tilde{\star} = \mu_0 + \sum_{r=1}^{\infty} \hbar^r \tilde{C}_r$
	be two such star products.
	Without loss of generality we can assume that $\star$ and $\tilde{\star}$
	coincide up to some order $k \geq 2$.
	Then we have 
	\begin{align*}
		\del C_{k+1}
		= -\frac{1}{2} \sum_{i=1}^{k} [C_i,C_{k+1-i}]
		= -\frac{1}{2} \sum_{i=1}^{k} [\tilde{C}_i,\tilde{C}_{k+1-i}]
		= \del \tilde{C}_{k+1}
	\end{align*}	   
	and thus $\del(C_{k+1}-\tilde{C}_{k+1}) = 0$.
	Choosing a homogeneous covariant derivative on $E$
	and using \autoref{thm:classicalHKR} we get
	$C_{k+1} - \tilde{C}_{k+1}
		= \hkr\bigl(\hkrInv(C_{k+1} - \tilde{C}_{k+1})\bigr)
		+ \del \hkrHomo(C_{k+1} - \tilde{C}_{k+1})$. 
	By \autoref{prop:HomoHKRTheorem} we know that all involved maps preserve the degrees of homogeneity,
	and therefore $\hkrInv(C_{k+1}-\tilde{C}_{k+1})$
	is a bivector field of homogeneous degree $-(k+1)$.
	Since $k \geq 2$ we thus know $\hkrInv(C_{k+1}-\tilde{C}_{k+1}) = 0$.
	Then
	$S = \id + \hbar^{k+1} \hkrHomo(C_{k+1} - \tilde{C}_{k+1})$
	is a homogeneous equivalence between $\star$ and $\tilde{\star}$ up to order $k+1$.  
\end{proof}
Recall that every homogeneous star product is homogeneously equivalent to a star product
with its associated Poisson bracket in first order, see \autoref{lem:EquivHomoStarProducts}.
Hence in order to compare two homogeneous star products
it is enough to consider homogeneous star products
which agree up to order $1$.

\begin{lemma}
	\label{lem:RelativeForm}
	Let $A \to M$ be a Lie algebroid.
	Let
	$\star = \mu_0 + \sum_{r=1}^{\infty}\hbar^r C_r$ and
	$\tilde\star = \mu_0 + \sum_{r=1}^{\infty}\hbar^r \tilde C_r$
	be a homogeneous star products on $A^*$ with $C_1 = \tilde C_1$.
	\begin{lemmalist}
		\item \label{lem:RelativeForm_1}
		Then 
		\begin{equation}
			\label{eq:RelativeForm}
			\tilde\Phi(\star, \tilde\star)
				\colon
			\Secinfty(A) \tensor \Secinfty(A)
				\ni
			s \tensor t \mapsto \iota^*(C_2 - \tilde C_2)^{-}\bigl(\Jmap(s),\Jmap(t)\bigr)
				\in
			\Cinfty(M)
		\end{equation}	 
		defines a  $\D_A$-closed $2$-form in $\Secinfty(\Anti^2 A^*)$. 
		\item \label{lem:RelativeForm_2}
		If $\star$ and $\tilde\star$ are homogeneously equivalent via a homogeneous equivalence
		transformation $S = \id + \sum_{r=1}^{\infty} \hbar^r S_r$,
		then there exists 
		$\alpha \in \Secinfty(A^*)$
		with $\Lie_{\alpha^\ver} = S_1$
		such that
		$\tilde{\Phi}(\star,\tilde\star) = \D_A \alpha$.
		\item \label{lem:RelativeForm_3}
		Let
		$\dbltilde\star = \mu_0 + \sum_{r=1}^{\infty}\hbar^r \dbltilde{C}_r$ be another homogeneous star product with $\dbltilde{C}_1 = C_1$.
		Then it holds
		\begin{equation}
			\tilde{\Phi}(\star,\tilde\star) + \tilde{\Phi}(\tilde\star,\dbltilde\star)
			= \tilde{\Phi}(\star,\dbltilde\star).
		\end{equation}
	\end{lemmalist}
\end{lemma}

\begin{proof}
	By the homogeneous HKR Theorem, cf. \autoref{prop:HomoHKRTheorem}, we know that
	$\hkrInv(C_2 - \tilde{C}_2) \in \Secinfty(\Anti^2 TA^*)$
	is homogeneous of degree $2$.
	Thus there exists
	$\tilde{\Phi}(\star,\tilde\star) \in \Secinfty(\Anti^2 A^*)$
	such that
	\begin{equation*}
		\label{eq:relativeForm_proof}
		\tag{$*$}
		\tilde{\Phi}(\star,\tilde\star)^\ver = \hkrInv(C_2 - \tilde{C}_2).
	\end{equation*}
	Since $\star$ and $\tilde\star$ agree up to order $1$ we have
	$\del(C_2 - \tilde{C}_2) = 0$,
	and thus by \autoref{lem:Skewstuff}	we have
	$\hkr\hkrInv(C_2 - \tilde{C}_2)
	= (C_2 - \tilde{C}_2)^-$.
	Thus applying $\hkr$ to \eqref{eq:relativeForm_proof} yields
	$\hkr\tilde{\Phi}(\star,\tilde\star)^\ver = (C_2 - \tilde{C}_2)^-$
	On the other hand, using the properties of the vertical lift,
	see \eqref{eq:VerticalLiftProps}, we obtain
	\begin{equation*}
		\hkr \tilde{\Phi}(\star,\tilde\star)^\ver\bigl(\Jmap(s),\Jmap(t)\bigr)
		= \pr^*\tilde{\Phi}(\star,\tilde\star)(s,t)
	\end{equation*}
	for all $s,t \in \Secinfty(A)$.
	Therefore, we have
	$\pr^*\tilde{\Phi}(\star,\tilde\star)(s,t)
	= (C_2 - \tilde{C}_2)^-(\Jmap(s),\Jmap(t))$.
	On the other hand we know that $\star$ and $\tilde\star$ agree up to order
	$1$ and thus $\del(C_2 - \tilde{C}_2) = 0$.
	This yields \eqref{eq:RelativeForm}.
	To show that $\tilde{\Phi}(\star)$ is $\D_A$-closed let
	$s, t, u \in \Secinfty(A)$
	be given, then
	\begin{align*}
		[\Jmap(s),[\Jmap(t),\Jmap(u)]_\star]_\star
		&= - \hbar^2\Jmap([s,[t,u]_A]_A)
		\\
		&\phantom{=}
		+\I \hbar^3 \pr^*\bigl(\rho_A(s)\iota^*(C_2 - \tilde{C}_2)^{-}(\Jmap(t),\Jmap(u))\bigr)
		\\
		&\phantom{=}
		+ \I \hbar^3 (C_2)^-\bigl(\Jmap(s),\Jmap([t,u]_A)\bigr).
	\end{align*}
	And the same holds for $\tilde\star$.
	The difference of the Jacobi identities of
	$[\argument,\argument]_\star$ and $[\argument,\argument]_{\tilde\star}$
	yields in order $\hbar^3$
	exactly
	$\D_A\tilde\Phi(\star,\tilde\star)(s,t,u) = 0$.

	For the second part let
	$S = \sum_{r=1}^\infty \hbar^r S_r$
	be a homogeneous equivalence between $\star$ and $\tilde\star$, i.e.
	$F \mathbin{\tilde\star} G = S^{-1}\bigl(S(F) \star S(G)\bigr)$
	for all $F,G \in \Cinfty(A^*)$.
	In first order of $\hbar$ this means that for 
	$F,G \in \Cinfty(A^*)$ we get
	\begin{equation*}
		\tilde{C}_1(F, G)
		= C_1(F,G) + (\del S_1)(F,G).
	\end{equation*}
	From $C_1 = \tilde{C}_1$ it follows that $\del S_1 = 0$,
	and thus $S_1$ is a vector field on $A^*$.
	Moreover, by assumption it is homogeneous of degree $-1$
	and therefore there exists $\alpha \in \Secinfty(A^*)$
	such that $S_1 = - \I \alpha^\ver$.
	Finally, for $s,t \in \Secinfty(A)$, we compute 
	\begin{align*}
		(\tilde{C}_2)^-\bigl(\Jmap(s),\Jmap(t)\bigr)
		&= C_2^-\bigl(\Jmap(s),\Jmap(t)\bigr)
			+ C_1^-\bigl(S_1(\Jmap(s)),\Jmap(t)\bigr)
			+ C_1^-\bigl(\Jmap(s), S_1(\Jmap(t))\bigr)
			- S_1 \bigl(C_1^-(\Jmap(s),\Jmap(t))\bigr)
		\\
		&= C_2^-\bigl(\Jmap(s),\Jmap(t)\bigr)
			+ \pr^*\rho_A(t)(\alpha(s))
			- \pr^* \rho_A(s)(\alpha(t))
			+ \pr^* \alpha([s,t]_A)
		\\
		&= C_2^-(\Jmap(s),\Jmap(t))
			+ \pr^*(\D_A\alpha)(s,t).
	\end{align*}
	and thus $\tilde\Phi(\star,\tilde\star) = \D_A \alpha$.
	The last part follows directly from
	\eqref{eq:RelativeForm}.
\end{proof}

\begin{proposition}[Relative class]
	\label{prop:relativeClass}
	Let $A \to M$ be a Lie algebroid.
	The map
	\begin{equation}
		\Phi \colon \Def_\homogen(A^*) \times \Def_\homogen(A^*) \to \Cohom^2(A),
		\qquad
		\Phi([\star],[\tilde\star])
		\coloneqq [\tilde\Phi(\star',\tilde{\star}')],
	\end{equation}
	where $\star'$
	and $\tilde{\star}'$ are any 
	homogeneous star products
	homogeneously equivalent to $\star$ and $\tilde\star$, respectively,
	such that they agree up to order $1$,
	is well-defined.
\end{proposition}
\begin{proof}
Using \autoref{lem:RelativeForm}~\ref{lem:RelativeForm_2}
and \ref{lem:RelativeForm_3}, 
we can see that if we have a star product $\star$ and two equivalent star products 
$\tilde\star$ and $\dbltilde\star$, such that all three coincide up to order $1$, then 
\begin{align*}
	[\tilde\Phi(\star,\tilde\star)] = [\tilde\Phi(\star,\dbltilde\star)].
\end{align*}
Moreover, since
$\tilde\Phi(\star,\tilde\star) = -\tilde\Phi(\tilde\star,\star)$,
the same holds true for the first argument.
Let now $\star'$ and $\tilde{\star}'$ be two star products 
which coincide up to order $1$, such that $\star'$ is equivalent to $\star$ 
and $\tilde{\star}'$ is equivalent to $\tilde\star$.
We can apply the equivalence of $\star'$
and $\star$ to $\tilde\star$ to obtain a star product $\tilde{\star}''$, which coincides with 
$\star'$
up to order $1$ and thus also with $\tilde{\star}'$.
Since $\star$ and $\tilde\star$ coincide up to 
order $1$, we obtain 
\begin{align*}
	\tilde{\Phi}(\star,\tilde\star)
	=
	\tilde{\Phi}(\tilde{\star},\tilde{\star}''). 
\end{align*}
Since, $\tilde{\star}'$ and $\tilde{\star}''$
are equivalent and coincide up to order one, we get 
\begin{align*}
	[\tilde\Phi(\tilde\star,\tilde\star')]
	= 
	[\tilde\Phi(\tilde\star,\tilde\star'')]
	= 
	[\tilde{\Phi}(\star,\tilde\star)]
\end{align*}
and thus the characteristic class is well-defined. 
\end{proof}
We call $\Phi([\star],[\tilde\star])$ the \emph{relative class}
of $\star$ and $\tilde\star$.
We will often write 
$\Phi(\star,\tilde{\star})$
instead of $\Phi([\star],[\tilde{\star}])$.
We can use this relative class to check if two given homogeneous star products are in fact homogeneously equivalent.

\begin{lemma}
\label{lem:RelativeClassEquivalence}
	Let $A \to M$ be a Lie algebroid
	and let $\star$ and $\tilde{\star}$ be two homogeneous star products on $A^*$.
	\begin{lemmalist}
		\item \label{lem:RelativeClassEquivalence_1}
		The homogeneous star products $\star$ and $\tilde\star$
		are homogeneously equivalent if and only if $\Phi(\star,\tilde{\star}) = 0$.
		\item \label{lem:RelativeClassEquivalence_2}
		Let $\nabla$ and $\nabla'$ be homogeneous covariant derivatives on $A^*$ and let 
		$B \in \Secinfty(\Anti^2 A^*)$ be closed.
		Then $\star_{\scriptscriptstyle(\nabla,B)}$ and $\star_{\scriptscriptstyle(\nabla',B)}$,
		as defined in \autoref{prop:ExistenceHomStarProd},
		are homogeneously equivalent.
	\end{lemmalist}
\end{lemma}

\begin{proof}
	From \autoref{lem:RelativeForm} we know that homogeneously equivalent star products satisfy
	$\Phi(\star,\tilde{\star}) = 0$.
	Thus assume $\Phi(\star) = \Phi(\tilde\star)$.
	We now need to show that $\star$ and $\tilde\star$ 
	are homogeneously equivalent.
	By \autoref{lem:EquivHomoStarProducts} it is enough to construct a homogeneous equivalence up to order $2$.
	We do this in two steps.
	First, using \autoref{lem:EquivHomoStarProducts} we can already assume that both star products coincide up to order $1$.
	Since $\Phi(\star,\tilde{\star}) = 0$
	we can find
	$\alpha \in \Secinfty(A^*)$, such that 
	$\tilde\Phi(\star, \tilde\star)= \D_A \alpha$.
	We consider the vertical lift $\alpha^\ver \in \Secinfty(TA^*)$ of $\alpha$ and define
	\begin{equation*}
		S \coloneqq \id - \hbar \Lie_{\alpha^\ver}.
	\end{equation*}
	Note that $S$ is indeed a homogeneous equivalence, see \autoref{sec:HomogeneityOnVB}.
	Let us define a new homogeneous star product by
	$F \mathbin{\star_{\scriptscriptstyle{S}}} G
	\coloneqq S^{-1} \bigl(S(F) \star S(G)\bigr)$
	for all $F,G \in \Cinfty(A^*)$.
	Since $\Lie_{\alpha^\ver}$ is a derivation we have
	$F \mathbin{\star_{\scriptscriptstyle{S}}} G
	= FG + \frac{\I\hbar}{2} \{F, G \}_\KKS + \sum_{r=2}^\infty \hbar^r B_r$.
	Moreover, one can compute
	\begin{align*}
		B_2(\Jmap(s),\Jmap(t))
		= \Lie_{\alpha^\ver} \{ \Jmap(s), \Jmap(t) \}_\KKS
			- \left\{ \Lie_{\alpha^\ver} \Jmap(s), \Jmap(t) \right\}_\KKS
			- \left\{ \Jmap(s), \Lie_{\alpha^\ver} \Jmap(t) \right\}_\KKS
	\end{align*}
	and thus
	\begin{align*}
		\big(B_2^{-} - C_2^{-} \big)(\Jmap(s),\Jmap(t))
		&= \Lie_{\alpha^\ver} \{ \Jmap(s), \Jmap(t) \}_\KKS
			- \left\{ \Lie_{\alpha^\ver} \Jmap(s), \Jmap(t) \right\}_\KKS
			- \left\{ \Jmap(s), \Lie_{\alpha^\ver} \Jmap(t) \right\}_\KKS
		\\
		&= \Lie_{\alpha^\ver}\Jmap([s,t]_E)
			- \left\{ \pr^* \alpha(s), \Jmap(t) \right\}_\KKS
			- \left\{ \Jmap(s), \pr^* \alpha(t) \right\}_\KKS
		\\
		&= \pr^* \alpha([s,t]_A)
			+ \pr^* \rho_A(t) (A(s))
			- \pr^* \rho_A(s) (A(t))
		\\
		&= - \pr^*(\D_A \alpha)(s,t).
	\end{align*}
	Thus we have shown
	$\tilde\Phi(\star_{\scriptscriptstyle{S}},\star)
	= - \D_A \alpha$,
	and in particular
	$\tilde{\Phi}(\star_{\scriptscriptstyle{S}},\tilde{\star})
	= \tilde{\Phi}(\star_{\scriptscriptstyle{S}},\star) + \tilde{\Phi}(\star,\tilde{\star})
	=0$.
	In a second step we want to relate $\star_{\scriptscriptstyle{S}}$ with $\tilde{\star}$.
	Using the Maurer-Cartan equation, we see that 
	\begin{equation*}
		\del(\tilde{C}_2-B_2)
		= \frac 12 \bigl([\tilde{C}_1,\tilde{C}_1]-[B_1,B_1] \bigr)
		= 0,
	\end{equation*}
	which implies that
	$\tilde{C}_2 - B_2
	= \hkr\bigl(\hkrInv(\tilde{C}_2 - B_2)\bigr) + \del \hkrHomo(\tilde{C}_2 - B_2)$
	by the homotopy formula for the Hochschild differential.
	For closed bidifferential operators, we know that the skew-symmetrization is a bivector field 
	and we have
	$\hkrInv(\tilde{C}_2 - B_2)
	= \tilde\Phi(\tilde{\star},\star_{\scriptscriptstyle{S}})^\ver 
	= 0$.
	This means we can use
	$\hkrHomo(\tilde{C}_2 - B_2) \in \Diffop(A^*)$
	to define 
	\begin{equation*}
		T \coloneqq \id + \hbar^2 \hkrHomo(\tilde{C}_2 - B_2).
	\end{equation*}  
	It is a homogeneous equivalence of $\tilde{\star}$ and $\star_{\scriptscriptstyle{S}}$ up to order to.
	Thus by \autoref{lem:EquivHomoStarProducts} there exists a homogeneous equivalence
	between $\star$ and $\tilde{\star}$.

	The second part follows directly from the third by observing that the anti-symmetric parts of the second order of  $\star_{\scriptscriptstyle(\nabla,B)}$ and $\star_{\scriptscriptstyle(\nabla',B)}$ are both given by $\hkr(B^\ver)$.
\end{proof}
By \autoref{prop:ExistenceHomStarProd} we know that we can construct for every closed Lie algebroid $2$-form
$B \in \Secinfty(\Anti^2 A^*)$ a homogeneous star product
$\star_{\scriptscriptstyle(\nabla,B)}$.
In particular, we can choose $B = 0$.
We can now define a characteristic class for every homogeneous star product by comparing it to $\star_{\scriptscriptstyle(\nabla,0)}$.

\begin{definition}[Characteristic class]
	\label{def:characteristicClass}
	Let $A$ be a Lie algebroid.
	We define the map $\Phi \colon \Def_\homogen \to \Cohom^2(A)$ by
	$\Phi([\star]) \coloneqq \Phi([\star],[\star_{\scriptscriptstyle(\nabla,0)}])$.
	For any star product $\star$ we call $\Phi([\star])$ its \emph{characteristic class}.
\end{definition}

Note that by \autoref{lem:RelativeClassEquivalence}~\ref{lem:RelativeClassEquivalence_2}
the definition of $\Phi$ is indeed independent of the choice of the homogeneous covariant derivative
$\nabla$.
Moreover, by definition of $\star_{\scriptscriptstyle(\nabla,0)}$ the characteristic class is given by
\begin{equation}
	\label{eq:HomCharClass}
	\Phi(\star)(s,t) = 
	\iota^*C_2^{-}\bigl(\Jmap(s),\Jmap(t)\bigr),
\end{equation}
for all $s,t \in \Secinfty(A)$.
Therefore one could equivalently define the characteristic class directly by \eqref{eq:HomCharClass}, without referring to the reference star product $\star_{\scriptscriptstyle(\nabla,0)}$.

\begin{corollary}
	\label{cor:CharClassOfConstruction}
	Let $A$ be a Lie algebroid.
	Moreover, let $\nabla$ be a homogeneous covariant derivative on $A^*$ and let $B \in \Secinfty(\Anti^2 A^*)$ be closed.
	Then $\Phi(\star_{\scriptscriptstyle(\nabla,B)}) = [B]$.
\end{corollary}
With this we finally arrive at the full classification of homogeneous star products.

\begin{theorem}[Classification]
	\label{thm:Classification}
	Let $A$ be a Lie algebroid.
	\begin{theoremlist}
		\item Two homogeneous star products $\star$ and $\tilde{\star}$ on $A^*$
		are homogeneously equivalent if and only if their characteristic classes coincide, i.e.
		$\Phi(\star) = \Phi(\tilde{\star})$.
		\item The map $\Phi \colon \Def_\homogen(A^*) \to \Cohom^2(A)$ is a bijection.
	\end{theoremlist}
\end{theorem}

\begin{proof}
	Note that
	$\Phi(\star,\tilde{\star})
	= \Phi(\star,\star_{\scriptscriptstyle(\nabla,0)}) - \Phi(\tilde{\star},\star_{\scriptscriptstyle(\nabla,0)})
	= \Phi(\star) - \Phi(\tilde{\star})$.
	With this, the first part is given by 
	\autoref{lem:RelativeClassEquivalence}.
	For the second we only need to show surjectivity.
	This was shown in \autoref{prop:ExistenceHomStarProd} together with \autoref{cor:CharClassOfConstruction}.
\end{proof}

\begin{corollary}
	\label{eq:HomoEquivIsEquiv}
	Let $A$ be a Lie algebroid.
	Two homogeneous star products $\star$ and $\tilde\star$ on $A^*$ are equivalent
	if and only if they are homogeneously equivalent.  
\end{corollary}
	
\begin{proof}
	To show the non-trivial implication let us assume that $\star$ and $\tilde{\star}$
	are (not necessarily homogeneous) equivalent homogeneous star products with equivalence $S = \id + \sum_{r=1}^\infty \hbar^r S_r$.
	We can assume that
	$C_1 = \frac{\I\hbar}{2} \{\argument,\argument\}_\KKS = \tilde{C}_1$.
	Then we know that $\del S_1 = 0$ and hence
	$S_1$ is a vector field.
	Moreover, we get for $s,t \in \Secinfty(A)$
	\begin{equation*}
		C_2^-\bigl(\Jmap(s),\Jmap(t)\bigr)
		= \tilde{C}_2^-\bigl(\Jmap(s), \Jmap(t)\bigr)
			- S_1\bigl(\Jmap ([s,t]_E ) \bigr)
			+ \bigl\{ S_1\bigl(\Jmap(s)\bigr), \Jmap(t)\bigr\}_\KKS 
			+ \bigl\{\Jmap(s),S_1\bigl(\Jmap(t)\bigr)\bigr\}_\KKS.
	\end{equation*}
	Let us define $\alpha \in \Secinfty(A^*)$ by 
	$\alpha(s) \coloneqq \iota^*S_1(\Jmap(s))$
	for $s \in \Secinfty(A)$.
	Then the above equation yields
	\begin{align*}
		\tilde\Phi(\star)(s,t)
		= \tilde\Phi(\tilde{\star})(s,t) + \D_A \alpha(s,t),
	\end{align*}
	and thus $\Phi(\star) = \Phi(\tilde{\star})$.
	Then \autoref{thm:Classification} shows that $\star$ and $\tilde{\star}$
	are actually homogeneously equivalent.
\end{proof}

\begin{remark}\
\begin{remarklist}
	\item Note that it is possible to classify homogeneous star products without 
	referring to a construction, see \eqref{eq:HomCharClass}.
	To our knowledge this has so far only been achieved in the symplectic case,
	see \cite{deligne:1995a}.
	\item \autoref{eq:HomoEquivIsEquiv} shows that the classification of homogeneous 
	tar products up to homogeneous equivalences agrees with the ordinary 
	classification of homogeneous star products seen as general star products.
\end{remarklist}
\end{remark}

\subsection{Characteristic Classes of Examples}
\label{sec:Comparison}
The construction from \autoref{sec:ClassificationHomStarProd} is rather easy 
compared to general constructions of star products. This is because of the homogeneity 
condition. On the other hand one can use general constructions to get 
homogeneous star products as well. We want to give a non-exhaustive overview 
of some constructions and compute the characteristic classes of the resulting 
homogeneous star products. 

\begin{example}[Neumaier-Waldmann construction]
In \cite{neumaier.waldmann:2009a} the authors constructed star products 
$\star_{(\nabla,B,\kappa)}$ quantizing the Poisson structure associated 
to a Lie algebroid $A$ out of the following data:
\begin{cptitem}
\item an $A$-connection $\nabla$,
\item a series $B=\sum_{k=0}^\infty \hbar^kB_k$ of $\D_A$-closed Lie algebroid $2$-forms and
\item a parameter $\kappa\in \mathbb{R}$.   
\end{cptitem}
It is shown in Proposition 4.1 of the same paper, that the star product is homogeneous, if $B_k=0$ for all $k\neq 1$. 
One can show additionally that for the choice $\kappa=\tfrac{1}{2}$ the resulting homogeneous star product possesses 
as a first order half of the Poisson bracket. 
Moreover having fixed $B_1$ star products for different $A$-connections
(Proposition 5.4 $iv.)$) and different $\kappa$ 
(Proposition 5.5 $iv.)$) are homogeneously equivalent.
Using the explicit formulas provided in \cite{neumaier.waldmann:2009a},
one can show that
$\tilde\Phi(\star_{(\nabla,\hbar B_1, \tfrac{1}{2})})
= B_1$ 
and thus the characteristic class is given by
$\Phi(\star_{(\nabla,\hbar B_1,\kappa)})
= [B_1]$.
\end{example}

\begin{example}[Kontsevich's Formality]
	\label{Subsec: quant}
	Kontsevich showed in \cite{kontsevich:2003a} that there is an
	$L_\infty$-quasi isomorphism 
	$K$ between $T_\poly^\bullet(\mathbb{R}^d)$ and $\HCdiff^\bullet(\mathbb{R}^d)$, 
	i.e. a sequence of maps 
	$K_n\colon T_\poly^n(\mathbb{R}^d) \to \HCdiff^n(\mathbb{R}^d)$
	of degree $1-n$ fulfilling certain compatibility conditions.
	Moreover, $K_1$ coincides with the Hochschild-Kostant-Rosenberg map $\hkr$.
	In \cite{dolgushev:2005b}, there is an explicit procedure to globalize 
	Kontsevich's quasi-isomorphism to arbitrary smooth manifolds. 
	In addition it is shown that if there is a Lie group $G$ acting on the
	manifold and there is is an invariant covariant derivative,
	the globalization $F^{\nabla}$ can be chosen to be equivariant.
	So let us choose a homogeneous covariant derivative $\nabla$ on $A^*$,
	which means that $F^\nabla$ is equivariant with respect to the Euler 
	vector field $\euler$.
	The bidifferential operators $C_r$ corresponding to the
	Kontsevich-Dolgushev star product corresponding to the linear
	Poisson structure $\pi_\KKS$ are computed by 
	\begin{align}
		C_r	= \frac{1}{r!}F^{\nabla}_r(\pi_\KKS\wedge\dots\wedge\pi_\KKS),
	\end{align}	   
	or equivalently 
	\begin{align}
		\star = \mu_0 
		+ \sum_{r=1}^\infty \frac{1}{r!} F^{\nabla}_r (\pi \wedge\dots\wedge \pi)
	\end{align}		
	for $\pi = \hbar\pi_\KKS$.
	This implies that 
	\begin{align}
		[\Lie_\euler,C_r]
		=
		\frac{1}{r!}F^{\nabla}_r\bigl(\Lie_\euler (\pi_\KKS\wedge\pi_\KKS\wedge\dots\wedge\pi_\KKS)\bigr)
		=
		- \frac{r}{r!}F^{\nabla}_r(\pi_\KKS\wedge\dots\wedge\pi_\KKS)
		=-r C_r.
	\end{align}
	This means in particular
	$\star = \mu_0 + \sum_{r=1}^\infty \hbar^r C_r$
	is a homogeneous star product.
	One can generally show that the Kontsevich star product for the
	Poisson structure $\pi$	is symmetric in the second order
	(see Corollary \ref{Cor: SkewKon}) and that 
	\begin{align}
		f \star g = fg + \frac{\hbar}{2}\{f,g\}_\KKS + \mathcal{O}(\hbar^2).
	\end{align}
	So the characteristic class of the Kontsevich star product is $0$.
	Choosing now a closed Lie algebroid $2$-form
	$B\in \Secinfty(\Anti^2A^*)$
	one can show that
	$\pi \coloneqq \hbar\pi_\KKS + \hbar^2B^\ver$
	is a formal Poisson structure with
	$(\hbar\frac{\del}{\del \hbar}+\Lie_\euler)(\pi) = -\pi$
	and thus the Kontsevich star product of $\pi$ is homogeneous. 
	Moreover, we can see that the anti-symmetrization of the second order term of
	the corresponding star product is given by $\hkr(B^\ver)$ and thus the
	characteristic class of this star product is
	$[B]\in \mathrm{H}^2(A)$.  
\end{example}

\begin{example}[The Universal Enveloping Algebra]
	\label{Ex: UEA}
	We follow \cite{moerdijk.mrcun:2003a} to construct the universal enveloping algebra of a Lie algebroid. 
	For a Lie algebroid $(A\to M,[\argument,\argument],\rho)$, 
	we first consider the sum $\Secinfty(A)\oplus\Cinfty(M)$. For a closed 
	$B\in \Secinfty(\Anti^2A^*)$, we define now the bracket
	\begin{align}
		[(a,f),(b,g)]_B \coloneqq ([a,b],\rho(a)g-\rho(b)f + B(a,b)),
	\end{align}	  
	which is a Lie bracket by the closedness of $B$. 
	We consider now the (reduced) universal enveloping algebra 
	$\cc{\mathcal{U}}(\Secinfty(A)\oplus\Cinfty(M))$ with the inclusion 
	$\iota\colon \Secinfty(E)\oplus\Cinfty(M)\to \cc{\mathcal{U}}(\Secinfty(A)\oplus\Cinfty(M))$ 
	and define 
	\begin{align}
		\mathcal{U}(A)=\cc{\mathcal{U}}(\Secinfty(A)\oplus\Cinfty(M))/I,
	\end{align}
	where $I$ is the ideal generated by the relation $\iota(0,f)\iota(a,g)\sim \iota(fa,fg)$ 
	for $f,g\in\Cinfty(M)$ and $a\in \Secinfty(A)$.
	Let us denote by $\circ_B$ the product of $\mathcal{U}(A)$ and choose an $A$-connection $\nabla$. 
	We define now recursively the map 
	$\mathrm{pbw}^\nabla\colon \Secinfty(\Sym^\bullet A) \to \mathcal{U}(A)$ by 
	\begin{align}
		\mathrm{pbw}^\nabla_0(f)
		=
		\iota(0,f) 
		\quad \text{ and } \quad 
		\mathrm{pbw}^\nabla_1(s)
		=
		\iota(s,0)
	\end{align}
	for $f \in \Cinfty(M) = \Secinfty(\Sym^0A)$
	and $s \in \Secinfty(A)=\Secinfty(\Sym^1A)$. 
	For $s_1 \vee \cdots \vee s_{k+1}\in \Secinfty(\Sym^{k+1}A)$, we define 
	\begin{align}
		\mathrm{pbw}^\nabla_{k+1}(s_1\vee \cdots\vee s_{k+1})
		=
		\frac{1}{k+1}\sum_{i=1}^{k+1} s_i\circ_B\mathrm{pbw}^\nabla_k(s_1\vee\overset{\wedge_i}{\dots}\vee s_{k+1})
		+
		\mathrm{pbw}^\nabla_k\bigl(\nabla_{s_i} (s_1\vee\overset{\wedge_i}{\dots}\vee s_{k+1})\bigr).
	\end{align}
	Note that this map actually fulfils $\mathrm{pbw}^\nabla(fs)=\iota(0,f)\circ_B\mathrm{pbw}^\nabla(s)$ for 
	all $s \in \Secinfty(\Sym A)$ and $f\in \Cinfty(M)$.
	In \cite{laurent-gengoux.stienon.xu:2021a} it is actually shown
	that this map is an isomorphism of $\Cinfty(M)$-modules.  
	Using this we can define the associative product 
	\begin{align}
		S\tilde\star_{\scriptscriptstyle{B}} T
		\coloneqq \sum_{i=0}^{k+l} \hbar^i \pr_{\Secinfty(\Sym^{k+l-i} E)}
		(\mathrm{pbw}^\nabla)^{-1}\bigl(\mathrm{pbw}^\nabla(s)
		\circ_B \mathrm{pbw}^\nabla(t)\bigr)
	\end{align}
	for $S\in\Secinfty(\Sym^k A)$ and $T\in\Secinfty(\Sym^lA)$.
	Moreover, we have 
	$S \tilde\star_{\scriptscriptstyle{B}} T=ST+\mathcal{O}(\hbar)$ 
	and the induced Poisson structure of
	$\tilde\star_{\scriptscriptstyle{B}}$ is actually 
	$\pi_\KKS$ restricted to polynomial functions.
	For $s,t\in \Secinfty(E)$ and $f\in\Cinfty(M)$, 
	we have additionally 
	\begin{align}
		[s,t]_{\star_{\scriptscriptstyle{B}}}
		= \hbar [s,t] + \hbar^2 B(s,t)
		\qquad\text{and}\qquad
		f \star_{\scriptscriptstyle{B}} s 
		= fs.
	\end{align}
	To the best of our knowledge there is no proof in the literature that the product $\tilde\star_{\scriptscriptstyle{B}}$
	is defined by bidifferential operators, i.e. that there is a homogeneous star product $\star_{\scriptscriptstyle{B}}$
	on $A^*$ which restricted to polynomial functions is given by $\tilde\star_{\scriptscriptstyle{B}}$. Anyway, 
	if we assume that this is the case one can show that its characteristic class is 
	$[B]\in \Cohom^2(A)$.
\end{example}

Note that in the case when the Lie algebroid is actually a 
Lie algebra we know that the construction 
in \autoref{Ex: UEA} actually gives bidifferential
 operators and therefore we can obtain a star product in this way. 
 This is the so-called \emph{Gutt star product}. In \cite{dito:1999a}
 it is proven that the constructions of Kontsevich and Gutt 
 are equivalent and the equivalence can be given by explicit 
 formulas.
 We believe that we can recover these formulas with our, 
 more general, results.

\section{Reduction of Homogeneous Star Products}
\label{sec:ReductionHomStarProd}

Every homogeneous star product $\star$ on a vector bundle $E_\Total \to M$ induces a linear Poisson
structure on $E_\Total$, see \autoref{lem:StarProdInducesLinearPoisson}.
Given a coisotropic submanifold of $E_\Total$ the question arises if $\star$ induces
a star product on the reduced space.

To be more precise, we say that a vector subbundle
$I \colon E_\Normal \hookrightarrow E_\Total$
over a submanifold $i \colon C \hookrightarrow M$
is a \emph{linear submanifold}.
Then a linear coisotropic submanifold $E_\Normal$ induces a 
linear distribution $D \subseteq T E_\Normal$
and a distribution $D_C \subseteq TC$.
If $D$ (and hence $D_C$) is simple, i.e. $E_\red \coloneqq E_\Normal / D$
is smooth, we we obtain a surjective submersion 
$P \colon E_\Normal \to E_\red$ of vector bundles covering
a surjective submersion
$p \colon C \to C / D \eqqcolon M_\red$.

Generalizing this situation, we consider in the following 
what we call \emph{constraint vector bundles}, i.e. a vector bundle
$E_\Total \to M$, together with a linear submanifold
$I \colon E_\Normal \hookrightarrow E_\Total$
and a surjective submersion of vector bundles
$P \colon E_\Normal \twoheadrightarrow E_\red$
covering a surjective submersion $p \colon C \to M_\red$.
In other words, we consider a diagram of the following form:
\begin{equation}
\begin{tikzcd}\label{Diag:ConstVectBundle}
	E_\Total
		\arrow[d]
	& E_\Normal
		\arrow[l,hook',"I"']
		\arrow[d]
		\arrow[r,twoheadrightarrow, "P"]
	& E_\red
		\arrow[d]\\
	M 
	& C
		\arrow[l,hook',"i"']
		\arrow[r,twoheadrightarrow, "p"]
	& M_\red
\end{tikzcd}.
\end{equation}
We will often abbreviate a constraint vector bundle by writing
$E_\Total \hookleftarrow E_\Normal \twoheadrightarrow E_\red$.

\begin{remark}
	The above definition of constraint vector bundles is equivalent to the definition of constraint vector bundles
	as in \cite{dippell.kern:2025a}.
\end{remark}
We will, as usual, denote by $\vanishing_{E_\Normal}$ the vanishing ideal of functions in $\Cinfty(E_\Total)$
that vanish on $E_\Normal$.
Moreover, we call 
\begin{equation}
	\normalizer_{E_\Normal}
	\coloneqq \{ F \in \Cinfty(E_\Total)\ 
		\mid \exists \tilde{F} \in \Cinfty(E_\red) \mid
		I^*F = P^*\tilde{F} \}
\end{equation}	
the \emph{generalized Poisson normalizer}.
Note that, $\normalizer_{E_\Normal} \subseteq \Cinfty(E_\Total)$
is a subalgebra and
$\vanishing_{E_\Normal} \subseteq \Cinfty(E_\Total)$
is a two-sided ideal.

\begin{definition}[Projectable star product]
	\label{def:ProjectableStarProd}
	Let $E_\Total \hookleftarrow E_\Normal \twoheadrightarrow E_\red$ be a constraint vector bundle.
	A homogeneous star product $\star$ on $E_\Total$ is called
	\emph{projectable} if
	\begin{definitionlist}
		\item $\normalizer_{E_\Normal}\formal{\hbar}$ is a subalgebra of
		$\Cinfty(E_\Total)\formal{\hbar}$,
		\item $\vanishing_{E_\Normal}\formal{\hbar}$
		is a left ideal in $\Cinfty(E_\Total)\formal{\hbar}$, and
		\item $\vanishing_{E_\Normal}\formal{\hbar}$
		is a two-sided ideal in $\normalizer_{E_\Normal}\formal{\hbar}$.
	\end{definitionlist}	
\end{definition}
The aim of this chapter is now to construct and classify projectable star products.
We call these star products projectable, since 
they allow a reduction (or projection) in a canonical way:
\begin{equation}
	\Cinfty(E_\red)\formal{\hbar}
	\cong (\normalizer_{E_\Normal}/ \vanishing_{E_\Normal})\formal{\hbar}\cong \normalizer_{E_\Normal}\formal{\hbar}/ \vanishing_{E_\Normal}\formal{\hbar}
\end{equation}
Note that the star product $\star_\red$ on $E_\red$ obtained in this way is by construction homogeneous as well.
Moreover, 
\begin{align}
	\Cinfty(E_\Normal)\formal{\hbar}
	= 
	(\Cinfty(E_\Total)/\vanishing_{E_\Normal})\formal{\hbar}
	=
	\Cinfty(E_\Total)\formal{\hbar}/\vanishing_{E_\Normal}\formal{\hbar}
\end{align}
is canonically a 
$(\Cinfty(E_\Total)\formal{\hbar},\star)-(\Cinfty(E_\red)\formal{\hbar},\star_\red)-$bimodule, where 
the left action $L$ is given by left multiplication and the right action $R$ is given by right multiplication.  
This means we get in particular that 
\begin{equation}
\begin{split}
	L_{F\star G}(1)
	&=[F\star (G\star 1)]_N
	= [F\star G]_N=R_{[G]_\red}([F]_N)
	\\
	&= R_{[G]_\red}(R_{[F]_\red}(1))
	= R_{[F]_\red\star_\red [G]_\red} (1),
\end{split}
\end{equation}
for $F,G\in \normalizer_{E_\Normal}$,
where we denoted by $[\argument]_N$ the equivalence class in the quotient 
$\Cinfty(E_\Total)/\vanishing_{E_\Normal}$ and by $[\argument]_\red$ the equivalence class in 
$\normalizer_{E_\Normal}/\vanishing_{E_\Normal}$. Using again more geometric language, 
this means for 
$F,G\in \normalizer_{E_\Normal}$ such that $\tilde{F},\tilde{G}\in \Cinfty(E_\red)$ with 
$I^*F=P^*\tilde{F}$ and $I^*G=P^*\tilde{G}$, we have 
\begin{align}
	I^*(F\star G)=P^*(\tilde{F}\star_\red \tilde{G}).
\end{align}

\subsection{Representable Star Products}
\label{sec:PresentableStarProd}

Before considering the problem of understanding general projectable star products,
we will in this section consider star products which are compatible only with a linear submanifold, in the sense that we can find a representation on the functions on the submanifold by differential operators.
Understanding this special case will be a crucial step to constructing and classifying projectable star products in \autoref{sec:ClassificationProjStarProd}.

\begin{definition}[Representable star product]
	\label{def:RepresentableStarProd}
A homogeneous star product
$\star = \sum_{r=0}^\infty \hbar^r C_r$ on $E_\Total$ is 
called \emph{representable} on a linear submanifold
$I \colon E_\Normal \hookrightarrow E_\Total$, 
if there exists a $\ComplexNum\formal{\hbar}$-linear differential operator
$\rho = \sum_{r=0}^{\infty} \hbar^r\rho_r \colon 
\Cinfty(E_\Total)\formal{\hbar} \to \Diffop(E_\Normal)\formal{\hbar}$,
with $\rho_r$ of homogeneous degree $-r$,
such that 
\begin{definitionlist}
	\item \label{def:RepresentableStarProd_1}
		$\rho_0(F)(G) = (I^*F) \cdot G$ for all
		$F\in \Cinfty(E_\Total)$ and $G \in \Cinfty(E_\Normal)$ and 
	\item \label{def:RepresentableStarProd_2}
		$\rho(F \star G) = \rho(F) \circ \rho(G)$
		for all
		$F, G \in \Cinfty(E_\Total)$.
\end{definitionlist} 
If \ref{def:RepresentableStarProd_2} only holds up to order $n \in \Naturals_0$ in $\hbar$ we call $\star$ \emph{representable up to order $n$}.  
\end{definition}
Note that if $\star$ is representable with representation $\rho$,
and $S$ is a homogeneous equivalence, then
$\tilde{\star}$ given by
$F \mathbin{\tilde{\star}} G
= S^{-1}\bigl(S(F) \star S(G)\bigr)$
is representable as well with representation
$\rho' = \rho \circ S$.
Moreover, every homogeneous star product is representable up to order $0$ by using
$\rho_0$ as defined in \autoref{def:RepresentableStarProd}~\ref{def:RepresentableStarProd_1}.

To study the existence of such homogeneous representations consider
$\Diffop(E_\Normal)$ as a $\Cinfty(E_\Total)$-bimodule.
Then we obtain the Hochschild complex
$\HC^\bullet\bigl(E_\Total,\Diffop(E_\Normal)\bigr)$,
see \autoref{sec:RepresentableHKR} for details.
Then direct computations yield the following result about the
extension of representations.

\begin{lemma}
	\label{lem:ExtendingReps}
	Let $\star = \sum_{r = 0}^\infty \hbar^r C_r$
	be a homogeneous star product on $E_\Total$ and let
	$I \colon E_\Normal \hookrightarrow E_\Total$
	be a linear submanifold.
	Additionally, assume that $\star$ is representable up to order
	$n \in \Naturals_0$ on $E_\Normal$ via 
	$\rho = \sum_{r=0}^{n} \hbar^r \rho_r$.
	\begin{lemmalist}
		\item \label{lem:ExtendingReps_1}
			The cochain $D_{n+1} \in \HC^2\bigl(E_\Total, \Diffop(E_\Normal)\bigr)$
			defined by
			\begin{equation}
				D_{n+1}(F,G) \coloneqq 
				\sum_{i=0}^{n} (\rho_i\circ C_{n+1-i})(F,G)
				-
				\sum_{i=1}^n \rho_{i}(F)\circ \rho_{n+1-i}(G),
			\end{equation}
			for $F,G \in \Cinfty(E_\Total)$,
			is closed.
		\item \label{lem:ExtendingReps_2}
			The cocycle $D_{n+1}$ is exact with $D_{n+1} = \del \rho_{n+1}$,
			if and only if $\star$ is representable up to order
			$n+1$ via $\rho' \coloneqq \rho + \hbar^{n+1} \rho_{n+1}$.
	\end{lemmalist}
\end{lemma}

\begin{proposition}
	\label{prop:RepOrdbyOrd}
Let $\star = \sum_{r = 0}^\infty \hbar^r C_r$
be a homogeneous star product on $E_\Total$ and let
$I \colon E_\Normal \hookrightarrow E_\Total$
be a linear submanifold.
\begin{propositionlist}
	\item \label{prop:RepOrdbyOrd_1}
	The star product $\star$ is representable up to order $1$ on $E_\Normal$ if and only if $E_\Normal$ is coisotropic.
	In this case we have
	$\rho_1(F)(I^*G) = I^*C_1(F,G)$.
\end{propositionlist}
Assume now that $\star$ is representable up to order $n \in \Naturals$ on $E_\Normal$ via 
$\rho = \sum_{\ell=0}^{\infty} \hbar^\ell \rho_\ell$.
\begin{propositionlist}[resume]
	\item \label{prop:RepOrdbyOrd_2}
	If $n=1$ it holds
	\begin{equation} \label{eq:RepresentableOrder1Condition}
		I^*C^-_{2}(F,G)
		+ \rho_1(\{F,G\}_\KKS)(1)
		- \bigl(\rho_1(F)\circ \rho_1(G)\bigr)(1)
		- \bigl(\rho_1(G)\circ\rho_1(F)\bigr)(1)
		=0
	\end{equation}	
	for all $F, G \in \vanishing_{E_\Normal}$,
	if and only if we can change $\rho$ in order $2$,
	such that $\star$ becomes representable up to order $2$.
    \item \label{prop:RepOrdbyOrd_3}
	If $n \geq 2$, then we can change $\rho$ in order $n+1$, 
    such that $\star$ becomes representable up to order $n+1$.
\end{propositionlist}
\end{proposition}  

\begin{proof}
In all three cases we show that $\hkrInv(D_{n+1}) = 0$.
Then $D_{n+1}$ is exact by \autoref{thm:RepresentableHKR}
and the representation can be extended to order $n+1$ by \autoref{lem:ExtendingReps}.
For the first part, note that
$\hkrInv(D_1)(\D F\at{E_\Normal}, \D G\at{E_\Normal}) 
= I^*\{F,G\}_\KKS$
for all $F,G \in \vanishing_{E_\Normal}$, see \autoref{sec:RepresentableHKR},
and thus $\hkrInv(D_1) = 0$ if and only if $E_\Normal$ is coisotropic.
In the case $n=1$ the condition \eqref{eq:RepresentableOrder1Condition}
directly ensures that $\hkrInv(D_2) = 0$.
Moreover, if $\star$ is representable up to order $2$,
we know from \autoref{lem:ExtendingReps}~\ref{lem:ExtendingReps_2} that $D_2 = \partial \rho_2$.
Thus we get $\hkrInv(D_2) = 0$, which is exactly \eqref{eq:RepresentableOrder1Condition}
by the definition of $\partial$, see \ref{eq:RepresentableDifferential}.
Finally, $n \geq 2$ the homogeneity of $\star$ yields
$\hkrInv(D_{n+1}) = 0$.
Thus we get the claim.
\end{proof}

Thus when searching for representable star products we can always assume that $E_\Normal$ is already coisotropic.
Then \ref{prop:RepOrdbyOrd_3} of the preceding Proposition directly yields the following result.

\begin{corollary}\label{Cor: RepOrd2}
Let $\star$ be a homogeneous star product on $E_\Total$ and let
$I \colon E_\Normal \hookrightarrow E_\Total$ be a 
linear coisotropic submanifold. 
Additionally, let $\star$ be representable up to order
$n\geq 2$ on $E_\Normal$, then $\star$ is representable. 
\end{corollary}

Recall that a homogeneous star product on $E_\Total$ induces the structure of a Lie algebroid on $A_\Total$
where $A_\Total^* = E_\Total$, see \autoref{lem:StarProdInducesLinearPoisson}.
A linear submanifold $E_\Normal\to E_\Total$ induces now a subbundle
\begin{equation}
	A_\Normal = \{ a\in i^\sharp A_\Total \mid \alpha(a) = 0 \text{ for all } \alpha\in E_\Normal\}
\end{equation}
of $A_\Total$ over $C\hookrightarrow M$
such that $E_\Normal = \Ann(A_\Normal)$.
If $E_\Normal$ is coisotropic, one can check fairly easily that
$I \colon A_\Normal\to A_\Total$
is a subalgebroid. 
Note that we obtain an induced map
$I^* \colon \Cohom(A_\Total) \to \Cohom(A_\Normal)$
between the respective Lie algebroid cohomologies.

\begin{lemma}
	\label{lem:adptCan-Form} 
	Let $A_\Total$ be a Lie algebroid with Lie subalgebroid
	$I \colon A_\Normal \hookrightarrow A_\Total$
	over a submanifold $i \colon C \hookrightarrow M$.
	Every homogeneous star product
	$\star$
	on $A_\Total^*$ 
	is homogeneously equivalent to a star product
	$\star' = \sum_{r = 0}^{\infty} \hbar^r C'_r$
	satisfying
	\begin{equation}
		I^*C'^-_2\bigl(\Jmap(s),\Jmap(t)\bigr)
		= \bigl(I^*\tilde\Phi(\star)\bigr)
		\bigl(i^\sharp s, i^\sharp t\bigr)
	\end{equation}
	for all $s,t \in \Secinfty(A_\Total)$
	with $i^\sharp s, i^\sharp t \in \Secinfty(A_\Normal)$
	and
	\begin{equation}
		\label{eq:lem:adptCan-Form}
		I^*C_1(\Jmap(s), \pr^*f) = 0,
	\end{equation}
	for all $s \in \Secinfty(A_\Total)$
	and $f \in \vanishing_C$.
\end{lemma}

\begin{proof}
By \autoref{lem:EquivHomoStarProducts} we can assume that
$C_1 = \frac{\I}{2} \{\argument,\argument\}_\KKS$.
Let us choose a complementary vector bundle $K \to C$
to $A_\Normal$
such that
$i^\#A_\Total \simeq A_N \oplus K$.
Let us moreover choose an $A_\Normal$-adapted Lie algebroid connection $\nabla$, i.e. 
a Lie algebroid connection $\nabla$ for $A_\Total$, such that for all
$s,t \in \Secinfty(A_\Total)$
with $i^\sharp s, i^\sharp t \in \Secinfty(A_\Normal)$
it holds
$i^\sharp(\nabla_s t) \in \Secinfty(A_\Normal)$.
Thus it induces a Lie algebroid connection $\nabla^\Normal$ 
on $A_\Normal$.
Note that by \autoref{sec:LieAlgebroids} the connection $\nabla$ can always be chosen 
such that $\tr(R^\nabla) = 0$ and $\tr(R^{\nabla^\Normal}) = 0$.
Additionally, we choose an adapted dual basis of $\Secinfty(A_\Total)$
consisting of
$e_i, f_j \in \Secinfty(A_\Total)$
and 
$e^i, f^j \in \Secinfty(A^*_\Total)$
such that
\begin{align*}
	i^\sharp e_i &\in \Secinfty\bigl(A_\Normal\bigr),
	& i^\sharp f_j &\in \Secinfty(K),
	\\
	i^\sharp e^i &\in \Secinfty\bigl(\Ann(A_\Normal)\bigr),
	& i^\sharp f^j &\in \Secinfty\bigl(\Ann(K)\bigr).
\end{align*}
Let us define the differential operator
$D^\nabla \colon \Cinfty(A_\Total^*) \to \Cinfty(A_\Total^*)$
of homogeneous degree $-1$ by 
\begin{align*}
	D^\nabla 
	\coloneqq \Lie_{(e^i)^\ver} \circ \Lie_{(e_i)^\hor} 
	- \Lie_{(f^j)^\ver} \circ \Lie_{(f_j)^\hor}.
\end{align*}
With this we can define
$S \coloneqq \id - \frac{\I\hbar}{2} D^\nabla$
and us it to define a new star product $\star'$
by
$F \star' G \coloneqq S^{-1}\bigl(S(F) \star S(G)\bigr)$.
Let now $s,t\in \Secinfty(A_\Total)$ be such 
that $i^\sharp s, i^\sharp t\in \Secinfty(A_\Normal)$, then we get 
\begin{align*}
	I^*\bigl(C_2^{'-}(\Jmap(s),\Jmap(t))\bigr)
	&= 
	I^*\bigl(C_2^-(\Jmap(s),\Jmap(t)) 
	+
	\frac{1}{2}D^\nabla\{\Jmap(s),\Jmap(t)\}_\KKS
	\\
	&\phantom{=}
	-
	\frac{1}{2}\{D^\nabla(\Jmap(s)),\Jmap(t)\}_\KKS
	-
	\frac{1}{2}\{\Jmap(s),D^\nabla(\Jmap(t))\}_\KKS\bigr)
	\\
	&= I^*\tilde\Phi(\star)(i^\sharp s, i^\sharp t)
		+\frac{1}{2}I^*\mathrm{tr} (R^\nabla)(i^\sharp s, i^\sharp t)
		-\mathrm{tr} (R^{\nabla^{N}})(i^\sharp s, i^\sharp t)\\
	&= I^*\tilde\Phi(\star)(i^\sharp s, i^\sharp t).
\end{align*}
Moreover, for the first order we obtain
\begin{align*}
	C'_1(\Jmap(s),\pr^*f)
	&= \frac{\I}{2} \bigl( \{\Jmap(s),\pr^*f \}_\KKS 
	+ \partial D^\nabla(\Jmap(s),\pr^*f) \bigr)
	\\
	&= \frac{\I}{2} \bigl( \{\Jmap(s),\pr^*f \}_\KKS
	+ e^i(s) \rho(e_i)(f) - f^j(s) \rho(f_j)(f) \bigr)
	\\
	&= \I e^i(s) \rho(e_i)(f),
\end{align*}
and thus $C'_1(\Jmap(s),\pr^*f) = 0$
for $f \in \vanishing_C$, since the vector fields $\rho(e_i)$ are tangential to $C$.
\end{proof}

This allows us to state and prove the following 

\begin{theorem}\label{Thm: Pull-backcharclass}
Let $A_\Total$ be a Lie algebroid
and $I \colon A_\Normal \hookrightarrow A_\Total$ a Lie subalgebroid.
Moreover, let $\star$ be a homogeneous star product on $A_\Total^*$
with characteristic class 
$\Phi(\star) \in \Cohom^2(A_\Total)$.
The star product $\star$ is representable on $\Ann(A_\Normal)$
if and only if $I^*\Phi(\star) = 0 \in \Cohom^2(A_\Normal)$. 
\end{theorem}

\begin{proof}
Let us assume
$I^*\Phi(\star) = 0 \in \mathrm{H}^2(A_\Normal)$.
By \autoref{lem:adptCan-Form} we may assume that $\star$ is already representable up to order $1$
and that 
$I^*\tilde{\Phi}(\star) = 0$.
Moreover, we have 
 $\rho_1(F)(1) = 0$ for all $F\in \Cinfty(A_\Total^*)$, which implies that 
\begin{equation}
	\rho_1\bigl(\{F,G\}_\KKS\bigr)(1)
	- \bigl(\rho_1(F) \circ \rho_1(G)\bigr)(1)
	- \bigl(\rho_1(G) \circ \rho_1(F)\bigr)(1)
	= 0,
\end{equation}
for all $F,G \in \Cinfty(A_\Total^*)$,
and
$I^*C_2^-(\Jmap(s), \Jmap(t))
= I^*\tilde{\Phi}(s\at{C}, t\at{C})
= 0$.
Together this shows that \eqref{eq:RepresentableOrder1Condition} is satisfied and hence
$\star$ is representable up to order $2$ by \autoref{prop:RepOrdbyOrd} and 
therefore representable by \autoref{Cor: RepOrd2}.

Conversely, let us assume that the star product
$\star = \sum_{r=0}^{\infty}\hbar^r C_r$ is representable via
$\rho=\sum_{k=0}^\infty\hbar^k\rho_k$.
Moreover, we can assume that $\star$ is given as in 
\autoref{lem:adptCan-Form}.
We define  
$\alpha \in \Secinfty(A_\Normal^*)$
by 
	\begin{align*}
		\alpha(s)=\rho_1\bigl(\Jmap(\tilde s)\bigr)(1)\at{C},
	\end{align*}
where $\tilde s\in \Secinfty(A_\Total)$ such that $\tilde s\at{C}= s$.
Let us check that this is well-defined:
First assume we have a $s \in \Secinfty(A_\Total)$, such that $s\at{C} = 0$, 
then we may find functions $f_i \in \vanishing_C$ and sections
$s^i \in \Secinfty(A_\Total)$, 
such that $s = f_i s^i$, this means 
\begin{align*}
	\rho_1\bigl(\Jmap(s)\bigr)(1)
	=
	\rho_1\bigl(\Jmap(f_i s^i)\bigr)(1)
	=
	\rho_1\bigl(\Jmap(s^i)\bigr)(I^*\pr^*f_i)
		+ I^*\Jmap(s_i) \bigl(\rho_1(\pr^*f_i)(1)\bigr)
		- I^*C_1(\Jmap(s_i), \pr^*f_i),
\end{align*}
where we exploited that $\rho$ is a representation in order $1$.
The first summand vanishes since 
$f_i \in \vanishing_C$, the second summand vanishes, since $\rho_1(p^*f)$ is a homogeneous
differential operator of degree $-1$ on $E^*$ and the third summand vanishes since $\star$ satisfies
\eqref{eq:lem:adptCan-Form}.
This means that the map $\alpha$ is well-defined. Similarly one 
can show that $\alpha(fs) = f\alpha(s)$ for all $f \in \Cinfty(C)$ and
$s\in \Secinfty(A_\Normal)$, which implies that $\alpha\in \Secinfty(A_\Normal^*)$.
We may find now an extension 
$\hat\alpha \in \Secinfty(A_\Total^*)$, such that
$I^*\hat\alpha = \alpha$.
We define now the homogeneous equivalence transformation
$S \coloneqq \id +\hbar\hat\alpha^\ver$ and consider the
homogeneous star product $\star'$ with 
\begin{align*}
	F \star' G \coloneqq S\bigl(S^{-1}(F) \star S^{-1}(G)\bigr)
\end{align*}
which is representable via $\rho' = \rho \circ S^{-1}$.
Note that for
$s \in \Secinfty(A_\Total)$
with 
$s \at{C}\in \Secinfty(A_\Normal)$, we get $\rho'_1(\Jmap(s))(1) = 0$, which using \autoref{prop:RepOrdbyOrd}~\ref{prop:RepOrdbyOrd_2} 
implies that 
\begin{align*}
	0
	= I^*C_2'^-(\Jmap(s),\Jmap(t))
	= (I^*\tilde\Phi(\star) - \D \alpha)(s\at{C},t\at{C})
\end{align*}
for all $s,t\in \Secinfty(A_\Total)$ with 
$s\at{C}, t\at{C} \in \Secinfty(A_\Normal)$.
Therefore $I^*\Phi(\star) = I^*\Phi(\star') = 0$.
\end{proof}

\subsection{Classification of Projectable Star Products}
\label{sec:ClassificationProjStarProd}

Let us now come back to projectable star products.
In contrast to representable star products not all star products which are homogeneously
equivalent to a projectable one will be projectable themselves.
This is intuitively clear, since a general homogeneous equivalence need not preserve 
the subalgebra $\normalizer_{E_\Normal}$ or the ideal $\vanishing_{E_\Normal}$,
but we will later on also give an explicit example, see \autoref{ex:NonEquivReducedStarProd}.
Thus let us define a specialized notion of equivalences.

\begin{definition}[Projectable equivalance]
	\label{def:ProjectableEquivalence}
	Let $E_\Total \hookleftarrow E_\Normal \twoheadrightarrow E_\red$ be a constraint vector bundle.
	Two projectable star products $\star$ and $\tilde\star$ on $E_\Total$
	are said to be \emph{projectably equivalent},
	if they are homogeneously equivalent via
	$S = \id + \sum_{r=1}^\infty \hbar^r S_r$, 
	such that
	\begin{align}
		S_i(\normalizer_{E_\Normal}) \subseteq \normalizer_{E_\Normal}
		\qquad\text{and}\qquad
		S_i(\vanishing_{E_\Normal}) \subseteq \vanishing_{E_\Normal}.
	\end{align}
	If $S$ is only an equivalence up to order $k$, we call $\star$ and $\tilde\star$
	\emph{projectably equivalent up to order $k$}.
\end{definition}
As a first consequence, we get immediately that
projectable equivalences induce homogeneous equivalences
on the reduced bundles.

\begin{lemma}
	Let $\star$ and $\tilde\star$ be homogeneous projectable star products on
	$E_\Total \hookleftarrow E_\Normal \twoheadrightarrow E_\red$
	which are projectably equivalent up to order $k$,
	then their reductions $\star_\red$ and $\tilde\star_\red$ on $E_\red$
	are homogeneously equivalent up to order $n$.
\end{lemma}

\begin{proof}
Let $S = \id + \sum_{r=1}^{\infty} \hbar^r S_r$ be a projectable equivalence between
$\star$ and $\tilde\star$. 
By \autoref{def:ProjectableEquivalence} every $S_r$ induces a homogeneous differential operator
$(S_r)_\red$ on $E_\red$.
Then $S_\red \coloneqq \id + \sum_{r=1}^{\infty} \hbar^r (S_r)_\red$
is a homogeneous equivalence between $\star_\red$ and $\tilde\star_\red$.
\end{proof}
In analogy to the case of homogeneous star products, we show now that projectable star products are always projectably equivalent up to order $1$, and it suffices to check projectable equivalence up to order $2$.
Compare this with \autoref{lem:EquivHomoStarProducts} and \autoref{prop:EquivHomStar}.

\begin{proposition}
	\label{prop:projectablePotential}
	Let $E_\Total \hookleftarrow E_\Normal \twoheadrightarrow E_\red$
	be a constraint vector bundle.
	Moreover, let $C \in \HCdiff^1(E_\Total)$
	be $\partial$-closed and symmetric.
	Moreover, assume that
	\begin{equation}
		\label{eq:projectablePotential_1}
		C(G_1,G_2) \in \normalizer_{E_\Normal}
		\qquad\text{and}\qquad
		C(F,H) \in \vanishing_{E_\Normal}
	\end{equation}
	for all $F \in \Cinfty(E_\Total)$,
	$G_1, G_2 \in \normalizer_{E_\Normal}$,
	$H \in \vanishing_{E_\Normal}$.
	Then there exists $D \in \HCdiff^0(E_\Total)$
	satisfying
	\begin{equation}
		\label{eq:projectablePotential_2}
		D(G) \in \normalizer_{E_\Normal}
		\qquad\text{and}\qquad
		D(H) \in \vanishing_{E_\Normal},
	\end{equation}
	for all $G \in \normalizer_{E_\Normal}$,
	$H \in \vanishing_{E_\Normal}$,
	such that $\partial D = C$.
\end{proposition}

\begin{proof}
 Since $C$ is symmetric we have 
\begin{align*}
	C\bigl(\vanishing_{E^*},\Cinfty(A_\Total^*)\bigr)
	= C\bigl(\Cinfty(A_\Total^*), \vanishing_{E^*}\bigr)
	\subseteq \vanishing_{E^*}.
\end{align*}
This means in particular that $C$ is a tangential bidifferential operator
and hence restricts to a bidifferential operator
$\hat{C}$ on $E_\Normal$.
By \eqref{eq:projectablePotential_1}
the bidifferential operator
$\hat{C}$ is 
projectable in the sense of \cite{dippell.esposito.schnitzer.waldmann:2024a:pre}.
By the projectable HKR-Theorem \cite{dippell.esposito.schnitzer.waldmann:2024a:pre},
$\hat{C}$ is exact since it is symmetric, i.e. there is a projectable differential operator 
$\hat{S}\in \Diffop(E^*)$, such that
$\del \hat{S} = \hat{C}$.
We extend $\hat{S}$ to a homogeneous differential operator $S$ on $E_\Total$.
Then 
$C - \partial S$ vanishes along $E_\Normal$ and by the usual
HKR-Theorem we find a homogeneous differential operator $T$,
which vanishes along $E_\Normal$
such that $C - \partial S = \partial T$.
Thus for $D \coloneqq S + T$ the claim follows.  	
\end{proof}

\begin{corollary}
	\label{cor:ProjectableEquivOrdern}
	Let $\star$ and $\tilde\star$ be projectable star products on
	$E_\Total \hookleftarrow E_\Normal \twoheadrightarrow E_\red$.
	\begin{corollarylist}
		\item The star products $\star$ and $\tilde\star$
		are projectably equivalent up to order $1$.
		\item If $\star$ and $\tilde\star$ are projectably
		equivalent up to order $n \geq 2$,
		then they are projectably equivalent up to order $n+1$.
	\end{corollarylist}
\end{corollary}

\begin{proof}
We can assume that $\star$ and $\tilde\star$ coincide up to order $n \neq 1$.
Then we get that
\begin{align*}
	\del (C_{n+1}-\tilde{C}_{n+1}) = 0.
\end{align*}
For $n=0$ we know that $(C_1 - \tilde{C}_1)^- = 0$.
For $n \geq 2$ we get
$(C_{n+1} - \tilde{C}_{n+1})^- = \hkr \hkrInv (C_{n+1} - \tilde{C}_{n+1})$
by \autoref{lem:Skewstuff}.
Now by homogeneity of $\star$ and $\tilde\star$ we know
$\hkrInv (C_{n+1} - \tilde{C}_{n+1}) = 0$
and thus also in this case we have
$(C_{n+1} - \tilde{C}_{n+1})^- = 0$.
This means that
$C_{n+1} - \tilde{C}_{n+1}$
is symmetric and therefore by 
\autoref{lem:Skewstuff} exact,
giving a projectable equivalence.
\end{proof}
The aim is now to classify all projectable star products in a way analogous to that
of homogeneous star products.
For this, let $\star$ be a projectable star product on a constraint vector bundle
$E_\Total \hookleftarrow E_\Normal \twoheadrightarrow E_\red$.
Since $\star$ is in particular homogeneous it induces a linear Poisson structure
$\{\argument,\argument\}_\KKS$ on $\Cinfty(E_\Total)$,
such that additionally $\normalizer_{E_\Normal} \subseteq \Cinfty(E_\Total)$
is a Poisson subalgebra and $\vanishing_{E_\Normal} \subseteq \normalizer_{E_\Normal}$
is a Poisson ideal.
In order to consider suitable characteristic classes for projectable star products
we need to realize this situation as a dual of some sort of Lie algebroid.
Thus consider a diagram
\begin{equation}
	\label{Diag:Liealg}
\begin{tikzcd}
	A_\Total
		\arrow[d]
	& A_\Normal
		\arrow[l,hook',"I"']
		\arrow[d]
		\arrow[r, twoheadrightarrow,"P"]
	& A_\red
		\arrow[d]\\
	M 
	& C
		\arrow[l,hook',"i"']
		\arrow[r,twoheadrightarrow,"p"]
	& M_\red
\end{tikzcd}
\end{equation}
consisting of a Lie algebroid $A_\Total \to M$, 
a subalgebroid $I \colon A_\Normal \to C$ over 
a submanifold $i \colon C \hookrightarrow M$
and a surjective submersion of Lie algebroids
$P \colon A_\Normal \twoheadrightarrow A_\red$
over a surjective submersion
$p \colon C \twoheadrightarrow M_\red$.
We call this triple a \emph{constraint Lie algebroid}
and often denote it by 
$A_\Total \hookleftarrow A_\Normal \twoheadrightarrow A_\red$.
We can "dualize" such a constraint Lie algebroid and obtain
\begin{equation}
\begin{tikzcd}
	A_\Total^*\arrow[d]
	& A^\vee_\Normal
		\arrow[l, "I^\vee"']
		\arrow[d]
		\arrow[r,"P^\vee"]
	& A^*_\red
		\arrow[d]\\
	M 
	& C
		\arrow[l,hook',"i"']
		\arrow[r,"p"]
	& M_\red
\end{tikzcd},
\end{equation}
where 
\begin{equation}
	A_\Normal^\vee \coloneqq \Ann(\ker P),
\end{equation}
see \cite{dippell.kern:2025a} for details.
Then $\Cinfty(A^*_\Total)$
becomes a Poisson algebra
with Poisson subalgebra $\normalizer_{A^\vee_\Normal}$
and a Poisson ideal $\vanishing_{A^\vee_\Normal} \subseteq \normalizer_{A^\vee_\Normal}$.
In fact, it is not hard to see that every constraint vector bundle 
$E_\Total \hookleftarrow E_\Normal \twoheadrightarrow E_\red$
with linear Poisson structure on $E_\Total$ being compatible with
$\normalizer_{E_\Normal}$ and $\vanishing_{E_\Normal}$ in the above sense is of the form
$A^*_\Total \hookleftarrow A^\vee_\Normal \twoheadrightarrow A^*_\red$.

Given a constraint Lie algebroid
$A_\Total \hookleftarrow A_\Normal \twoheadrightarrow A_\red$
we define the subcomplex $\Secinfty_\proj(\Anti^\bullet A^*_\Total)$
of the Lie algebroid complex $\Secinfty(\Anti^\bullet A^*_\Total)$
consisting of \emph{projectable forms} by
\begin{equation}
	\Secinfty_\proj(\Anti^k A^*_\Total)
	\coloneqq
	\{\alpha \in \Secinfty(\Anti^k A^*_\Total)
		\mid \exists\, \alpha_\red \in \Secinfty(\Anti^k A^*_\red)
	\text{ such that } I^*\alpha = P^*\alpha_\red\}. 
\end{equation}
Note that we obtain a maps
\begin{equation}
\begin{tikzcd}
	\Secinfty(\Anti^\bullet A^*_\Total)
		\arrow[r,hookleftarrow]
	& \Secinfty_\proj(\Anti^\bullet A^*_\Total)
		\arrow[r,"\red",twoheadrightarrow]
	&\Secinfty(\Anti^\bullet A^*_\red),
\end{tikzcd}
\end{equation}
leading to the following maps in cohomology:
\begin{equation}
\begin{tikzcd}
	\Cohom^\bullet(A_\Total)
	&\Cohom^\bullet_\proj(A_\Total)
		\arrow[l,"{[\argument]}"']
		\arrow[r,"\red"]
	&\Cohom^\bullet(A_\red)
\end{tikzcd}.
\end{equation}
The question arises what we can learn about a homogeneous star product on $A^*_\Total$
if we know that its characteristic class $\Phi(\star)$ admits a projectable
representative, i.e. there exists $B \in \Secinfty_\proj(\Anti^2 A^*_\Total)$
such that $\Phi(\star) = [B]$.
To answer this, we need to relate projectable star products to representable star products as discussed
in \autoref{sec:PresentableStarProd}.

\begin{remark}
	The notion of constraint Lie algebroid agrees with the one introduced in
	\cite{dippell.kern:2025a}.
	Moreover, what we call projectable forms is there given by the $\mathrm{N}$-component
	of the constraint Chevalley-Eilenberg complex.
\end{remark}

We consider the vector bundle map 
$A^\vee_\Normal \to A_\Total^* \times A_\red^*$ which covers the embedding
$C \hookrightarrow M \times M_\red$.
Having homogeneous star products on $A_\Total^*$ and on $A_\red^*$ one can build a homogeneous
star product on $A_\Total^*\times \cc{A}_\red^*$.
It turns out that the question of this star product being representable on $A^\vee_\Normal$ is connected to the projectability of the star product 
on $A_\Total^*$.
To be more precise, let $\star$ be a projectable 
star product on $A_\Total^*$ with reduction $\star_\red$. We define now the product 
\begin{align}
	(f_1\tensor g_1) \hat\star (f_2\tensor g_2)
	= f_1\star f_2 \tensor g_2\star_\red g_1	
\end{align}
on $\Cinfty(A_\Total^*) \tensor \Cinfty(A_\red^*)$. Note that we can extend it to 
$\Cinfty(A_\Total^* \times A_\red^*)$, since alle the maps are differentiable.
This star product is quantizing
$\pi_\KKS - (\pi_{\KKS})_{\red}$.
We can now define an algebra map
$\rho \colon \Cinfty(A_\Total^*) \tensor \Cinfty(A_\red^*) \to \Diffop(A_\Normal^\vee)$.
by
\begin{equation}
	\rho(F \tensor 1)(H)
	= I^* (F \star \tilde H)
	\qquad\text{and}\qquad
	\rho(1 \tensor G)(H) 
	= I^*(\tilde{H} \star \tilde{G}) 
\end{equation}
for $I^*\tilde{H} = H$ and 
$I^*\tilde{G} = P^*{G}$.
Again, this yields a well-defined
$\rho \colon \Cinfty(A^*_\Total \times A^*_\red) \to \Diffop(A_\Normal^\vee)$,
which one can easily see to be a representation of $\hat\star$.
Let us now apply \autoref{Thm: Pull-backcharclass}
to obtain computable characterisations of star products which are equivalent to projectable ones. 

\begin{theorem}
	\label{thm:ExProjStar}
Let $A_\Total \hookleftarrow A_\Normal \twoheadrightarrow A_\red$
be a constraint Lie algebroid
and let $\star$ be a homogeneous star product on $A_\Total^*$.
\begin{theoremlist}
	\item If there exists a closed $B \in \Secinfty_\proj(\Anti^2 A^*_\Total)$
	such that $\Phi(\star) = [B]$,
	i.e. there exists a closed $B_\red \in \Secinfty(\Anti^2 A^*_\red)$
	such that $I^*\Phi(\star) = P^*[B_\red]$,
	then $\star$ is homogeneously equivalent to a projectable
	star product $\star'$ with
	$\Phi(\star'_\red) = [B_\red]$.
	\item If $\star$ is projectable,
	then $I^*\Phi(\star) = P^*\Phi(\star_\red)$.  
\end{theoremlist}

\end{theorem}

\begin{proof}
Using \autoref{prop:ExistenceHomStarProd} we can find a homogeneous star product $\tilde\star$
on $A_\red^*$ with characteristic class $[B_\red]$. We consider now the star product 
$\star\tensor \tilde \star^\opp$ on $A_\Total^*\times \cc{A}_\red^*$. We use \autoref{Thm: Pull-backcharclass}, 
to check if $\star\tensor \tilde{\star}^\opp$ is representable on 
$A_\Normal^\vee\to A_\Total^*\times \cc{A}_\red^*$. 
One can now show that the associated Lie algebroid  to this coisotropic is in fact 
\begin{equation*}
	A_\Normal\overset{I\times -P}\longrightarrow A_\Total\times \cc{A}_\red. 
\end{equation*}
We may assume that $\star$ and $\tilde\star$ have the respective Poisson structures in first order,
and thus also 
$\star\tensor \tilde\star^\opp$ 
has the Poisson structure 
$\pi_\KKS - (\pi_{\KKS})_{\red}$ in first order.
Therefore
$\Phi(\star\tensor \tilde\star^\opp)
= \pr_{A_\Total}^*\Phi(\star)-\pr_{A_\red}^*[B_\red]$. 
Note that the minus sign appears since we consider the star product $\tilde{\star}^\opp$.
So now we compute 
\begin{equation*}
	(I\times -P)^*\Phi(\star\tensor \tilde\star^\opp) 
	= 
	(I\times -P)^*(\pr_{A_\Total}^*\Phi(\star)-\pr_{A_\red}^*[B_\red])
	=
	I^*\Phi(\star) - P^*[B_\red]=0.
\end{equation*}
Note that the seemingly missing minus sign in the last equality comes from the fact that 
$(-P)^* = (-1)^kP^* \colon \Secinfty(\Anti^k A_\red) \to \Secinfty(\Anti^k A_\Normal)$.
This means we find a homogeneous representation
$\rho$ of $\star\tensor \tilde{\star}^\opp$.
Let us denote by $\rho_\red$ the induced map and by
$\Cinfty(A_\red^*)\formal{\hbar} \to \Diffop(A_\Normal^\vee)\formal{\hbar}$
\begin{equation*}
	D_\red 
		\colon
	\Cinfty(A_\red^*)\formal{\hbar}
		\ni 
	F \mapsto \rho_\red(F)(1) 
		\in 
	\Cinfty(A_\Normal^\vee)\formal{\hbar}. 
\end{equation*}
We may find a series of homogeneous differential operators
$S =\id+\sum_{k\geq 1}\hbar^kS_k$,
such that 
$D_\red = S\circ (P^\vee)^*$
and we find a new representation of 
$\star\tensor \tilde\star^\opp$ by $\rho'= S^{-1}\rho S$, which fulfils now 
$\rho'_\red(F)(1) = (P^\vee)^*F$ and the map.
Let us now define the induced representation 
$\rho'_\Total\colon \Cinfty(A_\Total^*)\formal{\hbar} \to \Diffop(A_\Normal^\vee)\formal{\hbar}$
and 
\begin{align*}
	D_T
		\colon
	\Cinfty(A_\Total^*)\formal{\hbar} 
		\ni 
	F \mapsto \rho_\Total(F)(1) 
		\in 
	\Cinfty(A_\Normal^\vee)\formal{\hbar}. 
\end{align*}
We can find a series of homogeneous differential operators
$T = \id + \sum_k \hbar^k T_k$
on $A_\Total^*$, such that
$D_T(F) = (I^\vee)^*(T(F))$.
We define now the star product
$\star' = T(T^{-1}(\argument)\star T^{-1}(\argument))$
and 
$\rho''_T = \rho'_T\circ T^{-1}$.
Note that $\rho''_T$ is a representation for $\star'$, such that
$\rho''(F)(1)=(I^\vee)^*(F)$.
Moreover, we have that 
\begin{align*}
	\rho''_T(F)\circ \rho'_\red(G) 
	&= 
	\rho'_T(T^{-1}(F))\circ \rho'_\red(G)
	=
	\rho'(\pr_{A^*_\Total}^*(T^{-1}(F)) \star \tensor \tilde{\star}^{\opp}\pr_{A^*_\red}^*(G))
	\\
	&=
	\rho'(\pr_{A^*_\red}^*(G) \star\tensor\tilde{\star}^{\opp}\pr_{A^*_\Total}^*(T^{-1}(F)))
	=
	\rho'_\red(G)\circ  \rho''_T(F).
\end{align*}
Let us now assume that we have $F,G \in \normalizer_{A_\Normal^\vee}$ with
$\tilde{F},\tilde{G} \in \Cinfty(A_\red^*)$
such that
$(I^\vee)^*F = (P^\vee)^*\tilde{F}$ and
$(I^\vee)^*F = (P^\vee)^*\tilde{F}$,
then 
\begin{align*}
	(I^\vee)^*(F\star' G)
	&= \rho''_\Total(F\star' G)(1)
	= \rho''_\Total(F)(\rho''_\Total(G)(1))
	= \rho''_\Total(F)\bigl((I^\vee)^*(G)\bigr)
	\\
	&= \rho''_\Total(F)\bigl((P^\vee)^*(\tilde G)\bigr)
	= \rho''_\Total (F) (\rho'_\red(\tilde{G})(1))
	\\
	&= \rho'_\red(\tilde{G})(\rho''_\Total (F) (1))
	= \rho'_\red(\tilde{G}) (\rho'_\red(\tilde{F})(1))
	\\
	&= \rho'_\red(G\tilde\star^\opp F)(1)
	= (P^\vee)^*(\tilde{F}\tilde{\star} \tilde{G})
\end{align*}
Note that this shows that $\star'$ is projectable with reduction $\tilde{\star}$ and 
by construction
$\Phi( \tilde\star) = [B_\red]$. 

On the other hand if $\star$ is projectable with reduction $\star_\red$ then we can induce a 
representation as in the discussion right before \autoref{thm:ExProjStar} and see that the 
characteristic classes have to fulfil
$I^*\Phi(\star) = P^*\Phi(\star_\red)$
by the same arguments as in the beginning of the proof. 
\end{proof}

With this theorem we can now find an example of two projectable star products
which are homogeneously equivalent, but not projectably equivalent, and indeed
yield different reduction.

\begin{example}
	\label{ex:NonEquivReducedStarProd}
	Consider the diagram  
	\begin{equation}
		\mathbb{C}^{n+1} \simeq \mathbb{R}^{2n+2}
			\hookleftarrow
		S^{2n+1}
			\twoheadrightarrow
		\mathbb{C}P^n	
	\end{equation}
for $n\geq 1$ and apply the tangent functor to obtain a diagram of Lie algebroids:
    \begin{equation}
		T\mathbb{C}^{n+1} \simeq T\mathbb{R}^{2n+2}
			\hookleftarrow
		TS^{2n+1}
			\twoheadrightarrow
		T\mathbb{C}P^n.
	\end{equation}
We can consider the Weyl-Moyal star product on $T^*\mathbb{C}^{n+1}$ which is a homogeneous 
star product with characteristic class $0$,
see \autoref{ex:HomgeneousStarProds}~\ref{ex:HomgeneousStarProds_Moyal}.
Moreover, we consider the Fubini-Study symplectic
$\omega\in \Omega^2(\mathbb{C}P^2)$.
Then we can construct, using \autoref{thm:ExProjStar}, a star 
product $\star'$ which is equivalent to the Weyl-Moyal star product and reduces to a star product with characteristic
class $[\omega]$ (this is possible since the pull-back of $\omega$ to $S^{2n+1}$ is exact). On the other hand 
we find a star product $\star''$ which is equivalent to the Weyl-Moyal star product and reduces 
to a star product with characteristic class $0$. This means in particular that $\star'$ and 
$\star''$ are equivalent and projectable and have non-equivalent reductions. 
\end{example}
This example and the preceding theorem show that usual homogeneous equivalences are not
suited for studying the reduction of star products.
Thus instead of $\Def_\homogen(A^*_\Total)$ we want to study the set
$\Def_\proj(A^*_\Total)$ of projectable equivalence classes of projectable star products.
Analogously, we expect that $\Def_\proj(A^*_\Total)$ will not correspond to ordinary Lie algebroid cohomology, but to the cohomology
$\Cohom_\proj(A_\Total)$ of the projectable subcomplex.
Thus our situation is as follows:
we have the diagram
\begin{equation}
\begin{tikzcd}
		\Def_\homogen(A_\Total^*)
			\arrow[r,"\Phi","\simeq"']
		& \Cohom^2(A_\Total)
		\\
		\Def_\proj(A_\Total^*)
			\arrow[u,"{[\argument]}"]
			\arrow[d,"\red"']
		& \Cohom^2_\proj(A_\Total)
			\arrow[d,"\red"]
			\arrow[u,"{[\argument]}"']
		\\
		\Def_\homogen(A_\red^*)
			\arrow[r,"\Phi","\simeq"']
		& \Cohom^2(A_\red)
	\end{tikzcd},
\end{equation}
and our aim is to find an isomorphism
$\Def_\proj(A^*_\Total) \simeq \Cohom^2_\proj(A_\Total)$
such that both squares commute.
To do this we proceed as in the homogeneous case and first
consider the relative class of two projectable 
star products.

\begin{proposition}
	\label{prop:projRelativeForm}
Let $A_\Total \hookleftarrow A_\Normal \twoheadrightarrow A_\red$
be a constraint Lie algebroid,
and let $\star$ and $\tilde\star$ be projectable star products on $A^*_\Total$ which coincide up to order 
$1$.
\begin{propositionlist}
	\item Then $\tilde{\Phi}(\star,\tilde\star) \in \Secinfty_\proj(\Anti^2 A^*_\Total)$.
	\item 
	If $\star$ and $\tilde\star$ are projectably equivalent there exists 
	$\alpha \in \Secinfty_\proj(A^*_\Total)$
	such that
	$\tilde{\Phi}(\star,\tilde\star) = \D_A \alpha$.
\end{propositionlist}
\end{proposition}

\begin{proof}
By \autoref{eq:RelativeForm}~\ref{lem:RelativeForm_1}
we know that $\tilde{\Phi}(\star,\tilde\star) \in \Secinfty(\Anti^2 A^*_\Total)$.
Using the projectability of $\star$ and $\tilde\star$ and
evaluating on linear functions in $\normalizer_{A_\Normal^\vee}$
one can see that
$\tilde{\Phi}(\star,\tilde\star) \in \Secinfty_\proj(\Anti^2 A^*_\Total)$.
For the second part we know from \autoref{lem:RelativeForm}~\ref{lem:RelativeForm_2}
that there exists $\alpha \in \Secinfty(A^*_\Total)$
with $S_1 = \Lie_{\alpha^\ver}$
and $\tilde{\Phi}(\star,\tilde\star) = \D_A \alpha$.
Since the equivalence $S$ is projectable one can check that
$\alpha \in \Secinfty(A^*_\Total)$.
\end{proof}
Together with \autoref{cor:ProjectableEquivOrdern}
this allows us to define a projectable relative class.

\begin{definition}[Projectable relative class]
	\label{def:ProjRelativeClass}
	Let $A_\Total \hookleftarrow A_\Normal \twoheadrightarrow A_\red$
	be a constraint Lie algebroid.
	We define the map
	\begin{equation}
		\Phi_\proj
			\colon
		\Def_\proj(A^*) \times \Def_\proj(A^*)
			\to
		\Cohom^2_\proj(A),
		\qquad
		\Phi_\proj([\star],[{\tilde\star}])
			\coloneqq
		[\tilde\Phi(\star',\tilde{\star}')]_\proj,
	\end{equation}
	where $\star'$
	and $\tilde{\star}'$ are any 
	projectable star products
	projectably equivalent to $\star$ and $\tilde\star$ , respectively,
	such that they agree up to order $1$.
	We call $\Phi_\proj([\star],[\tilde\star])$ the
	\emph{projectable relative class}
	of $\star$ and $\tilde\star$.
\end{definition}

Using the same arguments as in \autoref{prop:relativeClass}, one can check that the projectable 
relative class is indeed well-defined.

\begin{lemma}
	\label{lem:ProjRelativeClassEquivalence}
	Let $A_\Total \hookleftarrow A_\Normal \twoheadrightarrow A_\red$
	be a constraint Lie algebroid.
	Two projectable star products $\star$ and $\tilde{\star}$ on $A^*_\Total$
	are projectably equivalent if and only if
	$\Phi_\proj(\star,\tilde\star) = 0$.
\end{lemma}

\begin{proof}
	One implication is given by \autoref{prop:projRelativeForm}.
	For the other implication suppose 
	$\Phi_\proj(\star,\tilde\star) = 0$.
	Thus there exists $\star'$, which coincides with $\star$ up to order $1$, 
	is equivalent to $\tilde\star$,
	and $\tilde{\Phi}(\star,\star') = 0$.
	Thus $C_2 - C'_2$ is a closed, symmetric
	bidifferential operator, fulfilling the properties
	of \autoref{prop:projectablePotential},
	and therefore exact with
	potential satisfying \eqref{eq:projectablePotential_2}.
	We use this potential as an equivalence, showing that
	$\star$ and $\star'$, and hence $\tilde\star$,
	are projectably equivalent up to order $2$.
	Then by \autoref{cor:ProjectableEquivOrdern}
	they are projectably equivalent.
\end{proof}

For homogeneous star products we were able to find a canonical choice for the first order of every homogeneous star product.
In the projectable setting this is not the case.
Nevertheless, we can choose a reference star product with certain properties, and this choice
will allow us to define a projectable characteristic class.
Thus let us now choose a projectable star product $\star_0$ with $\Phi(\star_0) = 0$, such that 
its reduction satisfies
$\Phi\bigl((\star_0)_\red\bigr) = 0$. 
Note that such a star product $\star_0$
always exists by \autoref{thm:ExProjStar}.
With this we can now define the map 
\begin{equation}
	\Phi_\proj^{\star_0}
		\colon
	\Def_\mathrm{proj}(A_\Total^*)
		\ni
	[\star]_\proj \mapsto \Phi_\proj(\star,\star_0)
		\in
	\Cohom^2_\mathrm{proj}(A_\Total),
\end{equation}
where $\Def_\mathrm{proj}(A_\Total)$ is the set of equivalence classes of projectable star products
modulo projectable equivalences. 

\begin{theorem}
	\label{thm:ClassCommutesWithRed}
Let $A_\Total \hookleftarrow A_\Normal \twoheadrightarrow A_\red$
be a constraint Lie algebroid.
For any projectable star product $\star_0$ on $A^*_\Total$ such that $\Phi(\star_0) = 0 = \Phi((\star_0)_\red)$, the map 
$\Phi_\proj^{\star_0} \colon \Def_\mathrm{proj}(A_\Total^*) \to \Cohom^2_\mathrm{proj}(A_\Total)$ is a bijection 
and makes the diagram 
	\begin{equation}
		\label{diag:QR=0}
	\begin{tikzcd}
		\Def_\homogen(A_\Total^*)
			\arrow[r,"\Phi"]
		& \Cohom^2(A_\Total)
		\\
		\Def_\proj(A_\Total^*)
			\arrow[u,"{[\argument]}"]
			\arrow[r,"\Phi_{\proj}^{\star_0}"]
			\arrow[d,"\red"']
		& \Cohom^2_\proj(A_\Total)
			\arrow[d,"\red"]
			\arrow[u,"{[\argument]}"']
		\\
		\Def_\homogen(A_\red^*)
			\arrow[r,"\Phi"]
		& \Cohom^2(A_\red)
	\end{tikzcd}
	\end{equation}
	commute. 
\end{theorem}

\begin{proof}
Let us start to show the commutativity of the upper square and let therefore $\star$ be a 
projectable star product.
Moreover, let $B \in \Secinfty_\proj(\Anti^2A^*_\Total)$ such that 
$\Phi_{\proj}^{\star_0}(\star) = [B]_\proj$.
Then $\star$ and $\star_{\scriptscriptstyle(\nabla,B)}$ are homogeneously equivalent since they have the same characteristic class.
Thus 
\begin{align*}
	\Phi(\star)
	= \Phi(\star_{\scriptscriptstyle(\nabla,B)})
	= [B]
	= [\Phi_{\proj,\star_0}(\star)].
\end{align*}
For the commutativity of the lower square let again $\star$ be a projectable star product with reduction $\star_\red$.
We may again assume that 
it coincides with $\star_0$ up to order $1$.
Since $\star$ and $\star_0$ coincide up to order 
$1$, also their reductions coincide up to order $1$, this means 
that 
$\Phi(\star_\red) = [\tilde{\Phi}\bigl(\star_\red,(\star_0)_\red\bigr)]$,
where we used that
$\Phi((\star_0)_\red) = 0$ by choice of $\star_0$.
One can check easily that 
$P^*\tilde{\Phi}(\star_\red,(\star_0)_\red)
= I^*\tilde{\Phi}(\star,\star_0)$, which shows the commutativity of the 
lower diagram.

Note that $\Phi_\proj^{\star_0}$ is injective by 
\autoref{lem:ProjRelativeClassEquivalence}.
Now we want to check that it is surjective.
Let therefore be
$B \in \Secinfty_\proj(\Anti^2 A_\Total^*)$
be closed.
Using \autoref{thm:ExProjStar} 
we can find a projectable star product $\star$ which has characteristic class
$[B]\in \Cohom^2(A_\Total)$
with $\Phi(\star_\red) = [B_\red]$.
Then there exists
$\alpha \in \Secinfty(A^*_\Total)$,
such that 
$\Phi_{\proj}^{\star_0}(\star)-[B]_\proj
=[\D_A\alpha]_\proj$.
Note that $[\D\alpha]_\proj$ is not necessarily $0$, since $\alpha$ is not necessarily 
in $\Secinfty_\proj(A^*_\Total)$, but we may assume that $I^*\D\alpha=0$, since 
$\Phi(\star_\red)=[B_\red]$.
Let us denote the reduction of $\star$ by $\star_\red$. 
We define the star product $\star^\alpha$ by 
$F\star^\alpha G \coloneqq S\bigl(S^{-1}(F)\star S^{-1}(G)\bigr)$.
with $S \coloneqq \id + \hbar \alpha^\ver$,
then we have that $\star$ and $\star^\alpha$ coincide up to order 
$1$ since $\alpha^\ver$ is a vector field and $(C^\alpha_2-C_2)^-=-\D_A\alpha$.
We now look at the star product $\star^\alpha\tensor \star_\red^\opp$ on 
$A_\Total^*\times A^*_\red$ and we define the map 
$\rho_1\colon \Cinfty(A_\Total^*)\to \Diffop(A_\Normal^\vee)$ by 
\begin{equation*}
	\rho_1\bigl(\pr_{A_\Total^*}^*(F)\bigr)(G)
	\coloneqq (I^\vee)^*C_1(F,\tilde G)
\end{equation*}
for $F\in \Cinfty(A_\Total^*)$ and 
$\tilde{G}\in \Cinfty(A_\Total^*)$ with $(I^\vee)^*\tilde G=G$
and
\begin{equation*}
	\rho_1 \bigl(\pr_{A_\red^*}^*(G)\bigr)(F)
	\coloneqq (I^\vee)^*C_1(\tilde{F},\tilde{G}),
\end{equation*}
where $(I^\vee)^*\tilde G = (P^\vee)^*G$
and $(I^\vee)^*\tilde F = F$.
Note moreover, that $\rho_1(F)(1)=0$ for all
$F\in \Cinfty(A_\Total^*\times A_\red^*)$.
Then
$\rho = \rho_0 + \hbar\rho_1$
is by definition a representation up to order $1$ of
$\star^\alpha \tensor \star_\red^\opp$,
since $\star$ is projectable with reduction $\star_\red$
and coincides with $\star^\alpha$ up to order $1$.
We want to apply \autoref{prop:RepOrdbyOrd}: let us denote by $C_i^{\alpha,\red}$ the the components of 
$\star^\alpha\tensor \star_\red$. Let $F,G\in \Cinfty(A_\Total^*)$ such that there exist
$\tilde F, \tilde G \in \Cinfty(A_\red^*)$ with $(I^\vee)^*F=(P^\vee)^*\tilde F$ and  $(I^\vee)^*G=(P^\vee)^*\tilde G$, therefore we have 
\begin{align*}
	\label{eq: Ctwovanishes}
	&\phantom{=}(I^\vee\times P^\vee)^* (C_2^{\alpha,\red})^
	-\bigl( \pr_{A_\Total^*}^*(F)-\pr_{A_\red^*}^*(\tilde F), \pr_{A_\Total^*}^*(G)-\pr_{A_\red^*}^*(\tilde F) \bigr)
	\\
	&= (I^\vee)^*(C^\alpha_2)^-(F,G)
		- (P^\vee)^*C_{\red,2}^-(\tilde F, \tilde G)
	\\
	&= (I^\vee)^*\D_A\alpha^\ver(F,G)
	= 0,
\end{align*}
since $I^*\D_A\alpha=0$.
Note that the vanishing ideal of $A^\vee_\Normal\to A_\Total^*\times A_\red^*$
is generated by functions of the form
$\pr_{A_\Total^*}^*(F)-\pr_{A_\red^*}^*(\tilde F)$ with 
$(I^\vee)^*F=(P^\vee)^*\tilde F$, which can be checked locally.
Moreover,by using the HKR-Theorem 
$(I^\vee\times P^\vee)^* (C_2^{\alpha,\red})^-$ is a section of the second exterior power of the 
normal bundle, thus \autoref{eq: Ctwovanishes} is enough to show that
$(I^\vee\times P^\vee)^* (C_2^{\alpha,\red})^-$ 
vanishes on the whole vanishing ideal. 

This means now that we can find a representation
$\rho=\sum_{k\geq 0} \hbar^k\rho_k$ of $\star^\alpha\tensor \star_\red^\opp$
on $A_\Normal^\vee$ with $\rho_1(F)(1)=0$ by using \autoref{prop:RepOrdbyOrd}. 
Now we repeat the proof from \autoref{thm:ExProjStar} to find an equivalence of 
$\star^\alpha$ to a projectable star product $\tilde{\star}^\alpha$, but the condition $\rho_1(F)(1)=0$ ensures that 
the equivalent star product $\tilde\star^\alpha$ has the same first order as $\star^\alpha$ and the 
same skewsymmetrization of the second order. This means that 
\begin{equation*}
	\Phi_\proj^{\star_0}(\tilde\star^\alpha)
	= \Phi_\proj(\tilde\star^\alpha,\star_0)
	= \Phi_\proj(\tilde\star^\alpha,\star)
		+ \Phi_\proj(\star,\star_0)
	= -[\D\alpha]_\proj + [B]_\proj +[\D\alpha]_\proj
	=[B]_\proj,
\end{equation*}
and thus $\Phi_\proj^{\star_0}$ is surjective.
\end{proof}

\begin{remark}[Quantization commutes with reduction]
Let us reinterpret \autoref{thm:ClassCommutesWithRed}
as a statement about quantization.
For this consider the inverse $\Phi^{-1}$ of the characteristic class map, assigning
to every $[B] \in \Cohom^2(A^*_\Total)$ an equivalence class of homogeneous star products
on $A^*_\Total$.
In \autoref{sec:Comparison} we have seen that this mapping can be understood as factoring
through various quantization procedures discussed there.
For example, we have seen that $\Phi^{-1}([B])$ is equal to the Kontsevich class
coming from the formal Poisson structure $\hbar \pi_\KKS + \hbar^2 B$.
With this in mind we can view the commutativity of the upper square in \eqref{diag:QR=0}
as saying that $(\Phi_\proj^{\star_0})^{-1}$ is a quantization that yields
projectable equivalence classes of projectable star products.
Then the commutativity of \eqref{diag:QR=0} says that one can either first reduce the projectable
class and then quantize using $\Phi^{-1}$, or first quantize using
$(\Phi_\proj^{\star_0})^{-1}$
and then reduce the class of projectable star products.
In other word, it says that quantization commutes with reduction. Moreover, note
that this statement does not depend on the choice of the reference star product 
$\star_0$. 
	
\end{remark}
\appendix
\section{Dolgushev-Kontsevich star product is symmetric in order $2$}
	\label{App: Kontsevichissym}
	We follow the notations and conventions of \cite{kraft.schnitzer:2024a}. Recall that the 
	globalized Kontsevich formality is 
	obtained by the zigzag
	 \begin{align}
		T_{\mathrm{poly}}(M)
		  \stackrel{\tau \circ \nu^{-1}}{\longrightarrow}
		  (\Omega(M,\mathcal{T}_\poly),D)
		  \stackrel{\mathcal{U}^B}{\longrightarrow}
		  (\Omega(M,\mathcal{D}_\poly),D+\del_M)
		  \stackrel{\tau\circ \nu^{-1}}{\longleftarrow}
		  (D_{\mathrm{poly}}(M),\del),
	  \end{align}
	where $\Omega(M,\mathcal{T}_\poly),D=\D+[B,\argument])$ and 
	$(\Omega(M,\mathcal{D}_\poly),D+\del_M=\D+ [B,\argument]+\del_M)$ are the Fedosov resolutions 
	of $T_{\mathrm{poly}}(M)$ and $D_{\mathrm{poly}}(M)$, respectively and $U^B$ is the fibre-wise 
	Kontsevich Formality twisted by the connection term $B$. We use the explicit quasi-inverse $P$ of  
	$\tau\circ \nu^{-1}$ constructed in \cite{kraft.schnitzer:2024a} to obtain the explicit 
	globalized Kontsevich Formality by 
	\begin{align}
		F=P\circ U^B \circ \tau \circ \nu^{-1},
	\end{align}
	where the first Taylor coefficient of this  $L_\infty$-quasi isomorphism is given by $\hkr$. 
	This means in particular that the second order of a star product  of a formal Poisson structure 
	$\hbar\pi=\sum_{k=1}^\infty\hbar^k\pi_k\in  \hbar\Secinfty(\Anti^2 TM)\formal{\hbar}$ obtained by the Kontsevich-Dolgushev 
	formality is given by  	
	\begin{align}
		F_1(\pi_2)+\frac{1}{2}F_2(\pi_1\wedge\pi_1)= & 
		F_1(\pi_2) + \frac{1}{2}P_2(U_1^B(\tau \circ \nu^{-1}(\pi_1))\wedge U_1^B(\tau \circ \nu^{-1}(\pi_1)))\\&
		+ \frac{1}{2}P_1(U_2^B(\tau \circ \nu^{-1}(\pi_1)\wedge\tau \circ \nu^{-1}(\pi_1))),
	\end{align}
	where the subscript of the maps indicates the Taylor coefficient of the corresponding $L_\infty$-morphism.
	The last term of the sum actually simplifies to 	
	\begin{align}
		P_1(U_2(\tau \circ \nu^{-1}(\pi_1)\wedge\tau \circ \nu^{-1}(\pi_1))),
	\end{align}
	since $B$ is one-form with values in fibre-wise formal vector fields and $P_1$ sets every element 
	coming from $\Omega^{\geq 1}(M,\mathcal{D}_\poly)$ to zero. Moreover, it is well-known that the 
	Kontsevich star product in $\mathbb{R}^d$ is symmetric in the second order (in fact even in every even order, see \cite{kontsevich:2003a}). 
	This implies that 
	$P_1(U_2^B(\tau \circ \nu^{-1}(\pi_1)\wedge\tau \circ \nu^{-1}(\pi_1)))$ is a symmetric bidifferential operator. 
	As explained in \cite{kraft.schnitzer:2024a} the explicit quasi-inverse $P$ has the feature that $P_k$ is only non-zero, 
	if the inserted terms have a total form-degree of $k-1$, which implies that 
	\begin{align}
		\frac{1}{2}P_2(U_1^B(\tau \circ \nu^{-1}(\pi_1))\wedge U_1^B(\tau \circ \nu^{-1}(\pi_1)))
		=
		P_2(U_2(B\wedge\tau \circ \nu^{-1}(\pi_1))\wedge U_1(\tau \circ \nu^{-1}(\pi_1))).
	\end{align}
	We take now a closer look to $U_2(B\wedge\tau \circ \nu^{-1}(\pi_1))$ which is a vector field and 
	bi-vector field inserted in  
	the second Taylor coefficient of the fibre-wise Kontsevich formality, so the only Kontsevich graphs which may 
	appear in the formula of the star product are 
	\begin{center}
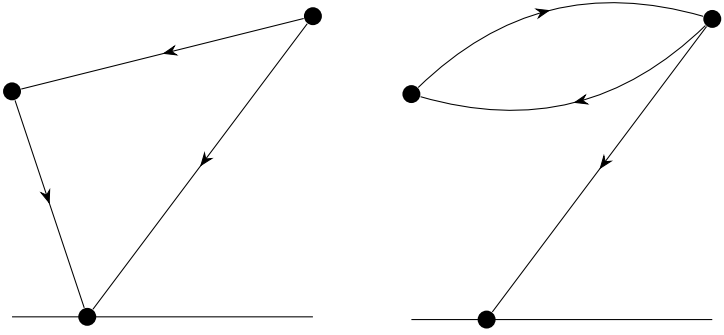
	\begin{minipage}{0.3\textwidth}
	 \begin{tikzpicture}[
		   decoration = {markings,
						 mark=at position .5 with {\arrow{Stealth[length=2mm]}}},
		   dot/.style = {circle, fill, inner sep=2.4pt, node contents={},
						 label=#1},
	every edge/.style = {draw, postaction=decorate}
							]
	\node (a) at (0,3) [dot];
	\node (b) at (1,0) [dot];
	\node (c) at (4,4) [dot];
	\path   (c) edge (a)    (c) edge (b)    
			(a) edge (b)   ;
	\draw[-] (0,0) -- (4,0);        
		\end{tikzpicture}
	\end{minipage}
	\begin{minipage}{0.3\textwidth}
	\begin{tikzpicture}[
		   decoration = {markings,
						 mark=at position .5 with {\arrow{Stealth[length=2mm]}}},
		   dot/.style = {circle, fill, inner sep=2.4pt, node contents={},
						 label=#1},
	every edge/.style = {draw, postaction=decorate}
							]
	\node (a) at (0,3) [dot];
	\node (b) at (1,0) [dot];
	\node (c) at (4,4) [dot];
	\path   (c) edge[bend left] (a)    (c) edge (b)    
			(a) edge[bend left] (c)   ;
	\draw[-] (0,0) -- (4,0);        
		\end{tikzpicture}
	\end{minipage}
	\end{center}
	The points on the horizontal lines are the exceptional vertices. Both graphs have the feature that there is one vertex with exactly one incoming 
	and exactly one outgoing arrow and due to \cite[Section 7.3.3.1]{kontsevich:2003a} the weights of these graphs vanish. This means in particular 
	
\begin{corollary}
		\label{Cor: SkewKon}
	Let $\hbar\pi=\sum_{k=1}^\infty\hbar^k\pi_k\in  \hbar\Secinfty(\Anti^2 TM)\formal{\hbar}$ be a formal Poisson structure. The skew symmetrization of 
	the second order of the Kontsevich-Dolgushev star product corresponding to $\hbar\pi$ is given by $F_1(\pi_2)$.   
\end{corollary}
	
\section{Representable HKR Theorem}
\label{sec:RepresentableHKR}

Following \cite[]{dippell.esposito.schnitzer.waldmann:2024a:pre}
we collect some results.
Let $E_\Total \to M$ be a vector bundle together with a linear submanifold
$I \colon E_\Normal \hookrightarrow E_\Total$ over a submanifold
$i \colon C \hookrightarrow M$,
and denote the normal bundle of $E_\Normal$ in $i^\#E_\Total$ by
$\nu(E_\Total,E_\Normal)$.
Consider the differential operators $\Diffop(E_\Normal)$ on the total space of the
vector bundle $E_\Normal$ as a $\Cinfty(E_\Total)$-bimodule with bimodule
structure given by
\begin{equation}
	(F \acts D)(G) \coloneqq I^*F \cdot D(G)
	\qquad\text{and}\qquad
	(D \racts F)(G) \coloneqq D(I^*F \cdot G),
\end{equation}
for $F \in \Cinfty(E_\Total)$ and $G \in \Cinfty(E_\Normal)$.
With this we obtain the Hochschild complex
$\HCdiff^\bullet\bigl(E_\Total,\Diffop(E_\Normal)\bigr)$
with differential given by
\begin{equation}
	\label{eq:RepresentableDifferential}
\begin{split}
	(\del D)(F_0, \dotsc, F_n)(G)
	&= 
	I^*F_0 \cdot D(F_1,\dotsc,F_n)(G)
	\\
	&\phantom{=}+ \sum_{i=0}^{n} (-1)^{i+1}
		D(F_0,\dotsc,F_i \cdot F_{i+1}, \dotsc, F_n)(G)
	\\
	&\phantom{=}+ (-1)^{n+1} D(F_0,\dotsc,F_{n-1})(I^*F_n \cdot G)
\end{split}
\end{equation}
for $F_0, \dotsc, F_n \in \Cinfty(E_\Total)$
and $G \in \Cinfty(E_\Normal)$.

\begin{theorem}[Representable HKR Theorem]
    \label{thm:RepresentableHKR}%
    Let $\nabla$ be a torsion-free homogeneous covariant
    de\-ri\-vative on $E_\Total$.
	Then there exist
	\begin{equation}
		\hkrHomo
			\colon
		\HCdiff^\bullet\bigl(E_\Total,\Diffop(E_\Normal)\bigr)
			\to
		\HCdiff^{\bullet-1}\bigl(E_\Total,\Diffop(E_\Normal)\bigr)
	\end{equation}
	and
	\begin{equation}
		\hkrInv
			\colon
		\HCdiff^\bullet\bigl(E_\Total,\Diffop(E_\Normal)\bigr)
			\to
		\Anti^{\bullet} \Secinfty\bigl(\nu(E_\Total,E_\Normal)\bigr)
	\end{equation}
	such that
    \begin{equation}
        \label{eq:representableHKRDeformationRetract}
        \begin{tikzcd}[column sep = large]
            \Anti^{\bullet} \Secinfty\bigl(\nu(E_\Total,E_\Normal)\bigr)
            	\arrow[r,"\hkr", shift left = 3pt]
            & \bigl( \HCdiff^\bullet\bigl(E_\Total,\Diffop(E_\Normal)\bigr),\del \bigr)
            	\arrow[l,"\hkrInv", shift left = 3pt]
            	\arrow[loop,
				out = -30,
				in = 30,
				distance = 30pt,
				start anchor = {[yshift = -7pt]east},
				end anchor = {[yshift = 7pt]east},
				"\hkrHomo"{swap}
				]
        \end{tikzcd}
    \end{equation}
    is a deformation retract, i.e. the following holds:
    \begin{theoremlist}
		\item \label{thm:representableHKR_1}
			We have $\hkrInv \circ \hkr = \id$.
		\item \label{thm:representableHKR_2}
			We have
			$\del \hkrHomo + \hkrHomo \del
			= \id - \hkr \circ \hkrInv$.
    \end{theoremlist}
	Moreover, the maps $\hkr$, $\hkrInv$ and $\hkrHomo$
	preserve the degrees of homogeneity.
\end{theorem}

In particular, for $D \in \HCdiff^2\bigl(E_\Total, \Diffop(E_\Normal)\bigr)$
we know that $\hkrInv(D) \in \Anti^2\Secinfty(\nu(E_\Total,E_\Normal))$,
and hence $\hkrInv(D)$ is determined by evaluating at elements in 
the annihilator of $E_\Normal$. Let us assume that $D$ is closed, then we get by
using the homotopy equation we get 
\begin{align}
	D(F,G)(1)=(\partial\hkrHomo D)(F,G)(1)+ \hkr\hkrInv(D)(F,G)(1)
\end{align}
for all $F,G\in \Cinfty(E_\Total)$. If we now assume that $F,G\in\vanishing_{E_\Normal}$, then we get 
by anti-symmetrizing
\begin{align}
	\hkr\hkrInv(D)(F,G)(1)= \frac{1}{2}\bigl(D(F,G)(1)-D(G,F)(1)\bigr),
\end{align}
since $(\partial\hkrHomo D)(F,G)(1)$ is symmetric evaluated on functions contained in 
the vanishing ideal. 
Moreover, using the explicit form of the map $\hkr$ on $F$ and $G$, this yields 
\begin{equation}
	\hkrInv(D)\left(\D F\at{C} , \D G\at{C}\right)
	= D(F,G)(1) - D(G,F)(1).
\end{equation}

\printbibliography[heading=bibintoc]

\newpage
\listoffixmes

\end{document}